\documentclass[a4paper,12pt,oneside,reqno]{amsart}
\usepackage{hyperref}
\usepackage[headinclude,DIV13]{typearea}
\areaset{15.1cm}{25.0cm}
\usepackage{amsfonts,amssymb,amsmath,amsthm,bbm}
\usepackage{cancel} 
\usepackage[latin1] {inputenc}
\usepackage{graphicx, psfrag}

\newtheorem{theorem}{Theorem}[section]
\newtheorem{lemma}[theorem]{Lemma}
\newtheorem{proposition}[theorem]{Proposition}
\newtheorem{corollary}[theorem]{Corollary}

\newtheorem{example}{Example}

\theoremstyle{definition}

\theoremstyle{remark}
\newtheorem*{remark}{Remark}

\def\paragraph#1{\noindent \textbf{#1}}

\numberwithin{equation}{section}

\def\d{\mathrm{d}}

\def\<{\langle}
\def\>{\rangle}
\def\a{\alpha}
\def\b{\beta}
\def\e{\epsilon}

\def\k{\kappa}

\def\R{{\Bbb R}}  
\def\N{{\Bbb N}}  
\def\Z{{\Bbb Z}}  
\def\Q{{\Bbb Q}}  
\def\C{{\Bbb C}}  

\def\H{{\Bbb H}}

\let\cal=\mathcal

 \def \k {{\kappa}}
 \def \b {{\beta}}
 \def \e {{\varepsilon}}

 \def \d {{\delta}}
 \def \a {{\alpha}}

 \def \ba {\begin{array}}
 \def \ea {\end{array}}

 \def \cB {{\cal B}}
 \def \cC {{\cal C}}
 \def \cD {{\cal D}}
 
 \def \cF {{\cal F}}
 
 \def \cH {{\cal H}}
 \def \cI {{\cal I}}
 \def \cL {{\cal L}}
 
 \def \cN {{\cal N}}
 \def \cP {{\cal P}}
 
  \def \cO {{\cal O}}

 \def \cR {{\cal R}}
 \def \cS {{\cal S}}

 \newcommand{\be}{\begin{equation}}
 \newcommand{\ee}{\end{equation}}

\newcommand{\bea}{\begin{eqnarray}}
 \newcommand{\eea}{\end{eqnarray}}
\def\TH(#1){\label{#1}}\def\thv(#1){\ref{#1}}
\def\Eq(#1){\label{#1}}\def\eqv(#1){(\ref{#1})}

 \def \1{\mathbbm{1}}

\begin{document}

\title[ Dimension Functions on the Spectrum over Bounded Geodesics]
{ Dimension Functions on the Spectrum over Bounded Geodesics and Applications to Diophantine Approximation }
\author[S. Weil]{Steffen Weil}
\address{S. Weil\\ School of Mathematical Sciences \\ Tel Aviv University \\ Tel Aviv 69978, Israel}
\email{steffen.weil@math.uzh.ch}


\begin{abstract} 
The set $\cB$ of geodesic rays avoiding a suitable obstacle in a complete negatively curved Riemannian manifold determines a spectrum $\cS$.
While various properties of this spectrum are known, we define and study dimension functions on $\cS$ in terms of the Hausdorff-dimension 
of suitable subsets of the set of bounded geodesic rays.
We establish estimates on the Hausdorff-dimension of these subsets and thereby obtain non-trivial bounds for the dimension functions.
Moreover we discuss the property of $\cB$ being an absolute winning set (see \cite{McMullen}), therefore satisfying a remarkable rigidity.
Finally, we apply the obtained results to the dimension functions on the  spectrum of complex numbers badly approximable by either an imaginary quadratic number field $\Q(i \sqrt{d})$ or by quadratic irrational numbers over $\Q(i \sqrt{d})$.
\end{abstract}

\maketitle


\section{Introduction and Results}
\label{Introduction}

\subsection{Outline}
Let $M$ be a complete connected Riemannian manifold of curvature at most $-1$.
The investigation of geodesics in $M$ avoiding an obstacle has been studied in various contexts
and is often deeply connected to problems in Diophantine approximation; see for instance \cite{PPPrescribing, HPDA}  and references therein.
Following the geometric viewpoint developed in these works (as well as in earlier ones such as \cite{HPSpiraling,  Patterson, Sullivan,Velani, VulakhDABianchi})
we continue the investigation as follows which will be made precise in the respective subsections below:

The geodesic flow $\phi^t : SM \times \R \to SM$ acts on the unit tangent bundle $SM$ of $M$. 
For a vector $v \in SM$ we call the orbit $\gamma_v(t) \equiv \pi \circ \phi^t(v)$, $t\geq 0$, a geodesic ray in $M$ (where $\pi :SM\to M$ denotes the footpoint projection).
Given an obstacle $O$ such as a 'cusp', a point, or a closed geodesic
call a geodesic ray in $M$ \emph{bounded} if it avoids a suitable neighborhood of the obstacle given in terms of a height, distance or length functional.
To each bounded geodesic ray $\gamma_v$ we assign a real constant $c(v)$, the \emph{approximation constant} defined by the respective functional.
The set $\cB$ of bounded geodesic rays starting in a point (or another set) determines the \emph{spectrum} $\cS \subset \R$.

Considering the modular surface $M=\H^2/PSL(2, \Z)$ and letting $O$ be the cusp of $M$,
$\cal{S}$ is related to the classical \emph{Markoff spectrum} $\cal{M}$.
Recall from  \cite{CusickFlahive} that the \emph{Lagrange spectrum} $\cL \equiv \{c^+(x)^{-1} : x \in \textbf{Bad}\} \subset \R$ and the 
Markoff spectrum $\cal{M} \equiv  \{ c(x)^{-1} : x \in \textbf{Bad}\} $ 
are determined by the approximation constants 
\be
\label{DefLagrange}
	c^+(x) \equiv \liminf_{p,q \to \infty} q^2\lvert x - \frac{p}{q} \rvert, \ \ \ \ c(x) \equiv \inf_{(p,q) \in \Z \times \N } q^2\lvert x - \frac{p}{q} \rvert 
\ee 
for \emph{badly approximable} numbers $x$ in $\textbf{Bad} = \{x \in \R :  c(x)>0\}$.
The spectrum $\cL$ is
\begin{itemize}
\item[1.] bounded below by the \emph{Hurwitz constant} $\frak{h} \equiv \inf \cL = \sqrt{5}$ (Hurwitz 1875),
\item[2.] contains a \emph{Hall ray} (Hall 1947, see below),
\item[3.] equals the closure of the set  $\{c^+(x)^{-1} : x \in \cal{P}\}$ where $\cP$ denotes the quadratic irrational numbers over $\Q$ (Cusick 1975); 
		in particular, $\cL$ is closed and
\be
\nonumber
	 \frac{1}{\frak{h}} = \sup_{x \in \cP} c(x).
\ee 
\end{itemize}
Moreover, the Lagrange spectrum is a subset of the Markoff spectrum 
where the inclusion is proper and  the intersection $\cal{L} \cap \cal{M}$ contains a positive half-line.

While various of the above properties were established also for the spectra of interest in our paper,
see for instance \cite{HPDA, Maucourant, HubertMarcheseUlcigrai,PPClosed, PPSpiraling, VulakhDABianchi},
the main intention of this paper is the following.
Define the \emph{dimension functions} $\mathfrak{D}$, $\frak{D}^0 : \cS \to \R$ on the spectrum $\cS$ via the Hausdorff-dimension of the 
sublevelset, respectively the levelset, of the assignment $c$.
We study these dimension functions for which we establish nontrivial bounds; see Section \ref{GeodesicFlow}.
For this we consider suitable subsets of the set of bounded geodesic rays and estimate  their Hausdorff-dimension.

Moreover, recall that \textbf{Bad}, the set of badly approximable numbers, is an \emph{absolute winning set} for the \emph{absolute game} (see McMullen \cite{McMullen}).
Absolute winning sets enjoy a remarkable rigidity.
In fact, an absolute winning set in $\R^n$ has full Hausdorff-dimension and is even \emph{thick}%
\footnote{ Recall that a subset $Y$ of a metric space $Z$ is thick if for any nonempty open set $O \subset Z$ we have that dim$(Y \cap O ) = $ dim$(Z)$. }
 in $\R^n$.
Moreover, an absolute winning set in $\R^n$ is preserved under quasi symmetric homeomorphisms and
a countable intersection of absolute winning sets  is absolute winning.
While the set $\cB$ of geodesic rays which are bounded in terms of a given cusp is an absolute winning set (see \cite{McMullen}), 
we establish and discuss analogue results for bounded geodesic rays, also in terms of the other obstacles.

Finally, exploiting the connection between the dynamics of geodesic rays in Bianchi orbifolds to 
Diophantine approximation of complex numbers badly approximable by imaginary quadratic number fields or by quadratic irrational numbers over such,
we obtain nontrivial bounds for the dimension functions on the corresponding spectra; see Section \ref{Applications}.
We conclude Section \ref{Introduction} by further discussion and, in order to keep the exposition readable, skip all the proofs of Section \ref{GeodesicFlow} to Section \ref{Proofs}.


\subsection{Bounded geodesic rays in negatively curved manifolds}
\label{GeodesicFlow}

Many of the following setups can be considered in a more general context, for instance when $M$ is geometrically finite, 
pinched or negatively curved, or even is a quotient of a proper geodesic CAT$(-1)$ metric space.
However, unless stated otherwise, 
we  assume for simplicity that $M$ is a complete $(n+1)$-dimensional finite volume hyperbolic%
\footnote{ By hyperbolic we mean constant negative sectional curvature $-1$.}
  Riemannian manifold.
As a general reference for the following see \cite{BGS} as well as Section \ref{Preliminaries}.

\subsubsection{Avoiding a cusp}
\label{SectionHeight}
Let $M$ be noncompact and let $e$ be a \emph{cusp} of $M$, that is to say an asymptotic class of minimizing geodesic rays along which the injectivity radius tends to $0$.
Let $\b_e$ be a Busemann function on $M$ (associated to such a minimizing geodesic ray) such that $H_t \equiv \beta_e^{-1}((t, \infty))$ 
gives shrinking  neighborhoods of the cusp as $t \to \infty$, which serves as a \emph{height function}. 
Up to renormalizing $\b_e$ assume that $H_{0}$ is a sufficiently small cusp neighborhood (see Section \ref{Preliminaries} for definitions).

Let $SH_0^+$ denote the $n$-dimensional submanifold of $SM$ consisting of outward unit vectors orthogonal to $\partial H_0$.
Each vector in $v \in SH_0^+$ can be identified with a geodesic line in $M$ starting from the cusp with $\gamma_v(0) \in \partial H_0$.
Define for a vector $v \in SH_0^+$ the \emph{height constant} to the cusp $e$ by
\be
\label{DefHeight}
	\cH(v) \equiv \sup_{t\geq 0}\beta_e(\gamma_v(t)) \in \R \cup \{\infty\}.
\ee
Note that for a typical $v \in SH_0^+$ we have that $\gamma_v$ is unbounded with $\cH(v)= \infty$.%
\footnote{ 
When $M$ has only one cusp, this follows for instance from Sullivan's \emph{logarithm law} \cite{Sullivan}: for almost all (spherical measure)
vectors $v \in SM_o$, where  $o$ in $M$ is a base point, we have
$\limsup_{t \to \infty} \frac{ d( \gamma_v(t) , o) }{\log(t)} = \frac{1}{n}$. %
}
Conversely a vector $v \in SH_0^+$ is called \emph{bounded} (with respect to the cusp $e$)  if $\cH(v) < \infty$.
By \cite{McMullen}, the set $\cal{B}_{M, e, \b_e} $ of bounded vectors  $v \in SH_0^+$  is of Hausdorff-dimension $n$ and in fact an absolute winning set; 
see \cite{MayedaMerrill} for further generalizations.
Define the  \emph{height spectrum} of the data $(M, e, \b_e)$ by  
\be
\nonumber
	\cS_{\cH} \equiv \{ \cH(v) : v \in SH_o^+ \text{ bounded} \} \subset [0, \infty). 
\ee
We define two \emph{dimension height functions}  $\frak{D}_{\cH}$, $\frak{D}_{\cH}^{0} : \cS_\cH \to [0,n]$ on the spectrum $\cS_\cH$ by
\bea
\label{DefDimFct}
	\frak{D}_{\cH}(t) &\equiv& \text{dim}(\{v \in SH_0^+ : \cH(v) \leq t\}), 
	\\ \nonumber
	\frak{D}_{\cH}^{0}(t) &\equiv& \text{dim}(\{v \in SH_0^+ : \cH(v) = t\}),
\eea
where 'dim' stands (here and hereafter) for the Hausdorff-dimension. 
Clearly, 
\be
\nonumber
	0 \leq \frak{D}_{\cH}^{0}(t) \leq \frak{D}_{\cH}(t) \leq n,
\ee
for all $t \in \cS_\cH$.
If $t \in \cS_\cH$ is a given height constant then $\frak{D}_{\cH}(t)$  equals the Hausdorff-dimension  of the set
\bea
\nonumber
	\cB_{M, e,\b_e}(t) &\equiv& \{v \in SH_0^+ :  \gamma_v(s) \not \in H_t \text{ for all } s\geq 0 \},
\eea
corresponding to the set of rays $\gamma_v$ avoiding the cusp neighborhood $H_t$ of $e$.

\begin{remark}
Consider the asymptotic height spectrum 
$\cS^+_\cH$ of $M$ (geometrically finite, negatively curved) instead, that is the spectrum of asymptotic height constants $\cH^+(v)$, where we use the 'limsup' in \eqref{DefHeight}, and restrict to positively recurrent vectors in $SH_0^+$;
a vector $v \in SM$ is \emph{positively recurrent} if the ray $\gamma_v$ hits a compact set $K$ in $M$ infinitely many times. 
Then the Properties 1. - 3. as above hold, where we replace $\cal{P}$ by the set of periodic vectors in $SM$, 
and  the Hurwitz constant can be determined explicitly in some concrete examples; see \cite{Maucourant, PPClosed, PPPrescribing, VulakhDABianchi} respectively.
Note that $\cH^+(v)\leq \cH(v)$ for every $v\in SH_0^+$.
Hence, defining the \emph{asymptotic dimension height function}  $\frak{D}_{\cH^+}$ in a similar way to \eqref{DefDimFct} with respect to $\cH^+$,
we obtain $\frak{D}_{\cH^+}(t) \geq \frak{D}_{\cH}(t)$  for all $t \in \cS_\cH\cap \cS_\cH^+$.
\end{remark}

\noindent From the author's earlier work \cite{Weil3}, when $M$ has only one cusp,
there exist a height $t_0$ and constants $k_u, k_l >0$ such that
for all $t_0 \leq t \in \cS_\cD$ we have
\be
\label{JarnikHeight}
	n - \frac{k_l}{t \cdot e^{n/2 t}} \leq \frak{D}_{\cH}(t) \leq n - \frac{k_u}{t \cdot e^{2n t}}.
\ee

\begin{remark}
In light of the correspondence between  badly approximable real numbers and bounded geodesic rays in  the modular surface $M=\H^2/PSL(2, \Z)$, 
\eqref{JarnikHeight} generalizes a classical inequality of Jarn\'ik \cite{Jarnik} and is called a \emph{Jarn\'ik-type inequality} by analogy in \cite{Weil3}.
A similar inequality holds when $M$ is geometrically finite, restricting to positively recurrent vectors.
\end{remark}

We next  establish non-trivial bounds for $\frak{D}_{\cH}^{0}(t)$.
When $M$ has precisely one cusp, each bounded geodesic ray $\gamma_v$, $v\in SH_0^+$, 
determines a countable discrete set of times $\{ t_i(v) : i \in \N\} \subset [0, \infty)$ of local maxima of the height function $\b_e$
with corresponding heights $h_i(v) = \b_e(\gamma_v( t_i(v) ))$.
If $M$ has more cusps, then possibly a subray of $\gamma_v$ may diverge to another cusp and we simply set  $t_i(v) = - \infty$ for sufficiently large $i$.
Given parameters $c_0\geq 0$ and $s_0\geq 0$, define the set of bounded vectors 
for which the first height $h_1$ equals $c_0$ and all others are bounded by $s_0$,
\be
\nonumber
	S(c_0, s_0) \equiv \{ v \in SH_0^+ :  h_1(v) = c_0 ,\ \  h_i(v) \leq s_0 \text{ for all } i \in \N_{\geq 2}\}
\ee
Schmidt, Sheingorn \cite{SchmidtSheingorn} showed for $n=1$
and Parkkonen, Paulin \cite{PPPrescribing} for $n\geq 2$ and curvature at most $-1$ 
that  $S(c_0, \bar s_0)$ is nonempty for all sufficiently large heights $c_0\geq \bar c_0$ and some constant $\bar s_0$.
When $S(c_0, c_0)$ is nonempty for all sufficiently large $c_0 \geq t_0$,  this implies the existence of a \emph{Hall ray}  at the cusp,
that is, there exists a height $t_0 \in \R$ such that $[t_0, \infty) \subset \cS_\cH$.%
\footnote{ 
Note that a bound on the height $t_0$ was determined explicitly with $t_0=4.16$ for $n=1$ in \cite{SchmidtSheingorn}, and with $t_0=4.2$ for $n\geq 2$ in \cite{PPPrescribing}.
}

Our first theorem establishes a lower bound on the dimensions of the sets $S(c_0, s_0)$.

\begin{theorem}
\label{ThmHeight}
Let $n\geq 2$.
There exists a height $t_0 \geq 0 $
and a positive constant $k_0>0$, both independent of $s_0$ and $c_0$,
such that for all heights $c_0$ and heights $s_0 \geq t_0$,
the Hausdorff-dimension of $S(c_0, s_0)$ is bounded below by
\be
\label{HeightLB}
	\text{dim}(S(c_0, s_0))
	\geq (n-1) -  \frac{ k_0  }{ s_0 }.
\ee
\end{theorem}

\begin{remark}
For $n=2$, the lower bound in \eqref{HeightLB} can be improved to  $1-  \frac{ k_0  }{ s_0\cdot e^{s_0/2}}$.
\end{remark}

Note that \eqref{HeightLB}  is trivially satisfied for $n=1$ and moreover that  $S(c_0, s_0)$ is nonempty whenever the lower bound in \eqref{HeightLB} is positive.
Thus, combining \eqref{JarnikHeight} and \eqref{HeightLB} we obtain the following.

\begin{corollary}
When $M$ has only one cusp, 
there exists a height $t_0\geq 0$ such that $[t_0, \infty) \subset \cS_\cH$  and positive constants $k_0, k_u>0$ such that for $ t\geq t_0$ we have
\be
\nonumber
	(n-1) -  \frac{ k_0  }{ t } \leq \frak{D}_{\cH}^0(t)  \leq  n - \frac{k_u}{t \cdot e^{2n t}}.
\ee
\end{corollary}

Finally, note from Section \ref{SectionAbsWinning} that $SH_0^+$ can be identified with the quotient of $\R^n$ by a discrete cocompact group $\Gamma_\infty$ acting on $\R^n$
such that the projection map
\be
\label{CanonicalMap}
	\R^n \ni x \mapsto [x] \equiv v_x \in \R^n/\Gamma_\infty = SH_0^+
\ee
is surjective and a local isometry.
Moreover, when $c_0>0$, the set $S_{c_0}$ of vectors $v \in SH^+_0$ (not necessarily bounded but defined in the same way as above) 
for which the first height $h_1$ equals $c_0$ can be identified with the quotient of $\tilde S_{c_0}/\Gamma_\infty$ 
where $\tilde S_{c_0} \subset \R^n$ consists of a countable disjoint union of $(n-1)$-dimensional Euclidean spheres in $\R^n$.
As a second theorem, lifting the set of bounded (with respect to $e$) vectors 
\be
\nonumber
	\cal{B}_{h_1=c_0} \equiv \cal{B}_{M,e,\b_e} \cap S_{c_0}  = \{ v_x \in \R^n/\Gamma_{\infty} : h_1(v_x)=c_0, \ \  h_i(v_x) < \infty \text{ for all } i \in \N_{\geq 2}  \}
\ee 
with first penetration height $c_0$ to $\tilde{\cal{B}}_{h_1=c_0} \subset \R^n$,
we show the following.

\begin{theorem}
\label{ThmAbsWinning}
When $n\geq2$, for each sphere $S$ in $\tilde S_{c_0}$ we have that $\tilde{\cal{B}}_{h_1=c_0}  \cap S $ is an absolute winning set in $S$.
\end{theorem}

\noindent 
For the definition of the absolute winning game we refer to Section \ref{SectionAbsWinningGame}.
An absolute winning set in a submanifold $S$ of $\R^n$ has full Hausdorff-dimension and is in fact thick in $S$;
for this and further properties of absolute winning sets in suitable subsets of $\R^n$ we refer to \cite{BroderickEtAl}.

\begin{remark}
The theorem also holds if $M$ is pinched negatively curved.
Moreover, it also follows from the  proof of the theorem that the intersection $\tilde{\cB}_{M, e, \b_e} \cap c(I)$ is absolute winning in $c(I)$ for any smoothly embedded (non-constant) curve $c: I \to \R^n$ where $I\subset \R$.
\end{remark}


\subsubsection{Avoiding a point}

Fix a point $x_0$  in $M$ which we view as obstacle, disjoint to a given base point $o$, and let  $d$ be the Riemannian distance function on $M$. 
Fix a technical constant $t_0 \geq 0$ and define for a given vector $v\in SM_o$
the \emph{distance constant} from the subray $\gamma_v\lvert_{[t_0, \infty)}$ to $x_0$ by 
\be
\nonumber
	\cD(v) \equiv  \sup_{t\geq t_0}  \big( -\log( d(\gamma_v(t), x_0)) \big). 
\ee
For a typical vector $v \in SM_o$ we have $\cD(v) = \infty$%
\footnote{ This follows from the logarithm law of \cite{MaucourantLogLaw}:
for almost  all vectors $v \in SM_o$ (with respect to the sphere measure on $SM_o$)  we have
$\limsup_{t \to \infty} \frac{- \log( d(\gamma_v(t), x_0) }{\log(t)} = \frac{1}{n}$.
 }
and we may call $v$ \emph{bounded} if $\cD(v) < \infty$.
By the author's earlier work \cite{Weil2}, the set of bounded vectors is of Hausdorff-dimension $n$ and in fact thick,
which will be improved below in Theorem \ref{PointAbsWinning}.
Define the   \emph{distance spectrum} of the data $(M, o, x_0, t_0)$ by    
\be
\nonumber
	\cS_{\cD} \equiv \{ \cD(v) : v \in SM_o \text{ bounded} \} \subset [-\log(t_0 + d(o,x_0)) ,\infty)
\ee
(for properties of the asymptotic distance spectrum, see Theorem \ref{SpectrumProperties} below).
Define as well the \emph{dimension distance function} $\frak{D}_{\cD} : \cS_\cD \to [0,n]$ 
by 
\be
\nonumber
	\frak{D}_{\cD}(t) \equiv  \text{dim}( \{ v \in SM_o : \cD(v) \leq t \}  ) = \text{dim}(\cB_{M,o,x_o,  t_0}(t)),
\ee
where 
\bea
\nonumber
	\cB_{M,o,x_o, t_0}(t) = \{v \in SM_o : \gamma_v(s) \not\in B(x_0, e^{-t}) \text{ for all } s\geq t_0\}
\eea
is the set of rays $\gamma_v\lvert_{[t_0, \infty)}$ avoiding the ball $B(x_0, e^{-t})$.

Our next theorem establishes a Jarn\'ik-type inequality as in \eqref{JarnikHeight} for the obstacle $x_0$.

\begin{theorem}
\label{ThmJTI}
When $M$ is compact, there exist a time $t_0\geq 0$, a distance $d_0$ and positive constants $k_u, k_l>0$  such that for all $d_0 \leq t \in \cS_\cD$  we have
\be
\nonumber
	n - \frac{k_l}{e^{n/2 t}} \leq  \frak{D}_{\cD}(t) \leq n - \frac{k_u}{t \cdot e^{n t}}.
\ee
\end{theorem}

\begin{remark}
Using the arguments of \cite{Weil3},
a Jarn\'ik-type inequality can be obtained when $M$ is convex-cocompact, when restricting to positively recurrent vectors in $SM_o$.
\end{remark}

Given a bounded vector $v \in SM_o$ consider the countable (possibly finite) discrete set $\{ t_i(v) : i \in \N \}\subset [t_0, \infty)$ of local minima for the distance function $t_0 \leq t \mapsto d(\gamma_v(t), x_0)$
and let $d_i(v) = d(\gamma_v(t_i(v)), x_0) $ be the corresponding distances;
note that we set $d_i(v)=1$ for all large $i$ if $\gamma_v$ eventually avoids the ball $B(x_0, 1)$.
Given the parameters $c_0, s_0 \in \R$, define the subset of bounded vectors 
\be
\nonumber
	S(c_0, s_0) \equiv \{ v \in SM_o:  d_1(v) = e^{-c_0} ,\ \  d_i(v) \geq e^{-s_0} \text{ for all } i \in \N_{\geq 2}\},
\ee
which is the set of rays $\gamma_v\lvert_{[t_0, \infty)}$ that have precisely distance $e^{-c_0}$ at time $t_1(v)$ (hence are tangent to the ball $B(x_0, e^{-c_0})$)
and avoid $B(x_0, e^{-s_0})$ for all $t \geq t_2(v)$.
Parkkonen, Paulin \cite{PPPrescribing} showed for $n\geq 2$ that $S(c_0, c_0)$ is nonempty for small $c_0 \leq - \log(2)$,
assuming a large injectivity radius of $M$ of curvature at most $-1$.

Our next theorem deals with large parameters $s_0, c_0$ and establishes a lower bound on the dimension.

\begin{theorem}
\label{ThmPoint}
Let $n\geq 2$.
There exists a time $t_0 \geq 0$, a distance $d_0 \in \R$ 
and a positive constant $k_0>0$ and  $k_1\geq0$, independent of $c_0$ and $s_0$,
such that for $c_0 \geq d_0$ and $s_0 \geq 2c_0 + k_1$,
the Hausdorff-dimension of $S(c_0, s_0)$ is bounded below by
\be
\nonumber
	\text{dim}(S(c_0, s_0))
	\geq (n-1) -  \frac{ k_0  }{ s_0 }.
\ee
\end{theorem}

\begin{remark}
Due to the condition that $s_0 \geq 2c_0$, Theorem \ref{ThmPoint} does not guarantee the existence of a Hall ray at the point $x_0$ (defined as for the case of a cusp).
We hope, however, that this condition can be relaxed.
\end{remark}

Finally, suppose  that $M$ is a complete geometrically finite Riemannian manifold of curvature at most $-1$.
Define for $v\in SM_o$ also the \emph{asymptotic distance constant}
\be
\nonumber
	\cD^+(v) \equiv  \limsup_{t \to \infty} \big( -\log( d(\gamma_v(t), x_0)) \big).
\ee
Let $SM_o^+$ denote the set of positively recurrent vectors  $v\in SM_o$ and 
note that if $v \not \in SM_o^+$ then $\cD^+(v) = - \infty$.
Denote by $\cS^+_\cD$  the \emph{asymptotic distance spectrum} consisting of the finite asymptotic distance constants $\cD^+(v)$ with $v\in SM_o^+$ which depends on the data $(M, x_0, o)$.
As for the asymptotic height and spiraling spectra below, using a result of Maucourant \cite{Maucourant}, we show the following properties.

\begin{theorem}
\label{SpectrumProperties}
The asymptotic distance spectrum $\cS^+_\cD$ is bounded below and $\cS^+_\cD$ equals the closure of the logarithmic distances $-\log(d(x_0, \alpha))$  from $x_0$ to closed geodesics $\a$ in $M$;
in particular, we have for the \emph{Hurwitz constant} $\frak{h}_\cD \equiv \inf \cS^+_\cD$ that 
\be
\nonumber
	e^{-\frak{h}_\cD} = \sup_{\a \emph{ a closed geodesic in } M } d(x_0, \a).
\ee
\end{theorem}

\begin{remark}
To the best of the author's knowledge, these properties do not already exist in the literature.
\end{remark} 

Clearly, $\cD^+(v) \leq \cD(v)$ so that if $v \in SM_o$ is bounded then it is asymptotically bounded.
Hence 
\be
\nonumber
	\cB_{M,o,x_o,t_0} \equiv \{v \in \overline{SM_o^+} : \cD(v) < \infty \} \subset \{v \in \overline{SM_o^+} : \cD^+(v) < \infty \},
\ee
and we remark that supsets of absolute winning sets are absolute winning.

\begin{theorem}
\label{PointAbsWinning}
Let $M$ be hyperbolic and geometrically finite. 
Then the set of bounded vectors $\cB_{M,o,x_o,t_0}$ is an absolute winning set in $\overline{SM_o^+}$.
In particular, when $M$ is convex-cocompact or of finite volume then $\cB_{M,o,x_o, t_0}$ is of full Hausdorff-dimension (and thick) in $\overline{SM_o^+}$.
\end{theorem}

\begin{remark}
The above result does not follow from previous results on points badly approximable by limit points in the hyperbolic space such as \cite{Patterson}.

When $M$ is of finite volume, it follows from the same proof that $ \cB_{M,o,x_o, t_0} \cap c(I)$ is absolute winning in $c(I)$ for any smoothly embedded (non-constant) curve $c: I \to SM_o$ where $I\subset \R$.
\end{remark}


\subsubsection{Avoiding a closed geodesic.}
\label{SectionLength}
Fix a closed geodesic  $\alpha_0$ in $M$.
For geodesic rays that avoid the obstacle $\a_0$  appropriate neighborhoods of $\a_0$  should in fact be given in the unit tangent bundle $SM$.
We therefore follow \cite{HPSpiraling, PPSpiraling} and consider the closed  $\e_0$-neighborhood $\cN_{\e_0}(\a_0)$ of $\a_0$ in $M$ where $\e_0>0$ is sufficiently small with respect to $\a_0$.
When $\gamma$ has bounded penetration lengths in the neighborhood $\cN_{\e_0}(\a)$ of $\alpha$ in $M$
then $\dot \gamma$ avoids a small neighborhood of $\dot \alpha$ (depending on the penetration lengths) in $SM$.
Hence, given a geodesic $\gamma$ in $M$, define its \emph{penetration length} at time $t$ by
$\cL_{\a_0, \e_0}(\gamma, t)= 0$ if $\gamma(t) \not \in \cN_{\e_0}(\a_0)$ and otherwise by
$\cL_{\a_0, \e_0}(\gamma, t) \equiv \ell(I)$,
where $\ell(I)$ denotes the length of the maximal connected interval $I\subset \R$ such that  $t \in I$ and  $\gamma(s) \in \cN_{\e_0}(\a_0)$ for all $s\in I$.

Fix again a base point $o \in M$ with $o \not \in \cN_{\e_0}(\a_0)$.
Using the terminology from \cite{PPSpiraling}, 
define for $v \in SM_o$ the \emph{spiraling constant} of $\gamma_v$ in $\cN_{\e_0}(\alpha_0)$ by
\be
\label{DefLength}
	\cL(v) \equiv \sup_{t \geq 0} \cL_{\a_0, \e_0}(\gamma_v,t). 
\ee
When $M$ is compact, a typical vector $v \in SM_o$ satisfies $\cL(v)= \infty$%
\footnote{
This follows from the logarithm law of \cite{HPSpiraling}: 
for almost all vectors $v \in SM_o$  (sphere measure) we have
$\limsup_{t \to \infty} \frac{ \cL_{\a_0, \e_0}(\gamma_v,t) }{\log(t)} = \frac{1}{n}$.
}
and we call $v$ \emph{bounded} when each possible penetration length in $\cN_{\e_0}(\a_0)$ is bounded above by $\cL(v)< \infty$.
However, we remark that (even for negative curvature or when convex-cocompact) the set of bounded vectors $v \in SM_o$ is of Hausdorff-dimension $n$ and in fact an absolute winning set; see \cite{Weil2}, also for further generalizations.

Define the  \emph{spiraling spectrum}  of the data $(M, o, \a_0, \e_0)$ by
\be
\nonumber
	\cS_{\cL} \equiv \{ \cL(v) : v \in SM_o  \text{ bounded} \} \subset [0, \infty).
\ee

\begin{remark}
Using a different setup, \cite{PPSpiraling} showed that the asymptotic spiraling spectrum%
\footnote{ That is the spectrum of the asymptotic heights $\cL^+(v)$, $v \in SM_o$, where we use the 'limsup' in \eqref{DefLength}. }
$\cS^+_\cL$ satisfies Properties 1. - 3. above where we replace $\cal{P}$ by the set of periodic vectors in $SM$. 
\end{remark}

\noindent By analogy to \eqref{DefDimFct}, define the \emph{dimension spiraling functions} $\frak{D}_{\cL}$, $\frak{D}_{\cL}^0: \cS_\cL \to [0,n]$ on the spectrum $\cS_\cL$.
Let $t \in \cS_\cL$ be a given length and denote by
\be
\nonumber
	\cB_{M, o, \a_0, \e_0}(t) \equiv \{v \in SM_o : \cL(v) \leq t \}
\ee
the set of rays $\gamma_v$ with spiraling constants bounded above by length $t$ such that $\frak{D}_{\cL}(t)= \text{dim}(\cB_{M, o, \a_0, \e_0}(t))$.
From the author's earlier work \cite{Weil3}, when $M$ is compact, there exist a length $t_0\geq 0$ and constants $k_u, k_l >0$ such that
for all $t_0 \leq t \in \cS_\cL$ we have 
\be
\label{JarnikLength}
	n - \frac{k_l}{t \cdot e^{n/2 t}} \leq \frak{D}_{\cL}(t)  \leq n - \frac{k_u}{t \cdot e^{n t}}.
\ee

\begin{remark}
A similar inequality holds when we replace the $\e_0$-neighborhood of $\a_0$ by the one of a
 higher-dimensional (up to codimension one) totally geodesic submanifold which is $(\e_0, T)$-immersed (see Section \ref{Preliminaries} or \cite{PPPrescribing}  for a definition)
or for $M$ convex-cocompact; see \cite{Weil3} for further details.
\end{remark}

In the following, we  establish nontrivial bounds for $\frak{D}_{\cL}^{0}(t)$.
Each bounded vector $v \in SM_o$ determines a sequence of countably many discrete penetration times $t_i(v) \geq 0$ 
and penetration lengths $l_i(v) = \cL_{\a_0, \e_0}(\gamma_v, t_i(v))>0$ 
such that $\gamma_v( [t_i(v), t_i(v) + l_i(v)]) \subset \cN_{\e_0}(\a_0)$;
note that we set $l_i(v)=0$ for all large $i$ if $\gamma_v$ eventually avoids $\cN_{\e_0}(\a_0)$.
Given $c_0\geq 0$ and $s_0\geq 0$, define the set of bounded vectors for which the first penetration length $l_1(v)$ equals $c_0$ 
and all others are bounded above by $s_0$,
\be
\nonumber
	S(c_0, s_0) \equiv \{ v \in SM_o :  l_1(v) = c_0 ,\  l_i(v) \leq s_0 \text{ for all } i \in \N_{\geq 2}\}.
\ee
When $n\geq 2$, Parkkonen, Paulin \cite{PPPrescribing} 
showed that $S(c_0, \bar s_0)$ is nonempty for all sufficiently large lengths $c_0\geq \bar c_0$ and some constant $\bar s_0$.
Note that when $S(c_0, c_0)$ is nonempty for all sufficiently large $c_0 \geq t_0$,  this implies the existence of a \emph{Hall ray}  at the closed geodesic $\a_0$,
that is, there exists a length $t_0 \geq 0$ such that $[t_0, \infty) \subset \cS_\cL$.

Our next theorem establishes a lower bound on the dimension of this set.

\begin{theorem}
\label{ThmLength}
Let $n\geq 2$.
There exists a length $t_0 \geq \log(2)$ 
and a positive constant $k_0>0$, independent of $s_0$ and $c_0$,
such that for all lengths $c_0, s_0 \geq t_0$,
the Hausdorff-dimension of $S(c_0, s_0)$ is bounded below by
\be
\label{LBLenght}
	\text{dim}(S(c_0, s_0))
	\geq (n-1) -  \frac{ k_0  }{ s_0 }.
\ee
\end{theorem}

\begin{remark} 
A similar lower bound holds when we replace the $\e_0$-neighborhood of $\a_0$ by the one of a higher-dimensional (up to codimension $2$) totally geodesic submanifold 
which is $(\e_0, T)$-immersed. 
\end{remark}

Combining \eqref{JarnikLength} and \eqref{LBLenght} we obtain the following corollary.
\begin{corollary}
Let $M$ be compact.
There exists a length $t_0\geq 0$ such that $[t_0, \infty) \subset \cS_\cL$ and positive constants $k_0, k_u>0$  such that for $ t\geq t_0$ we have
\be
\nonumber
	(n-1) -  \frac{ k_0  }{ t } \leq \frak{D}_{\cL}^0(t) \leq  n - \frac{k_u}{t \cdot e^{n t}}.
\ee
\end{corollary}


\subsection{Applications to Diophantine approximation and further discussion}
\label{Applications}
We will now shortly discuss  applications of Sections \ref{SectionHeight} and \ref{SectionLength} to Diophantine approximation.
For further background and details, we refer to \cite{ElstrodtEtAl, PPPrescribing} and references therein.

\subsubsection{Imaginary quadratic number fields}
For a positive square free  integer $d$ let $\cO_{d}$ be the ring of integers in the imaginary quadratic number field $\Q(i \sqrt{d})$.
For a complex number $z \in \C $, define its \emph{approximation constant} by
\be
\label{DefApproxConstant}
	c_{d}(z) \equiv \inf_{(p,q) \in \cO_{d} \times (\cO_{d} \setminus \{0\}) }   \lvert q\rvert^2 \lvert z - \frac{p}{q} \rvert .
\ee
When $c_d(z)>0$ we call $z$ \emph{badly approximable} by $\Q(i \sqrt{ d })$ and 
denote  \textbf{Bad}$_d \equiv \{ z \in \C : c_d(z)>0\}$ the set of badly approximable complex numbers.
Let $\cS_d$ be the \emph{spectrum} of logarithmic approximation constants $-\log(c_d(z))$,  $z \in \textbf{Bad}_d$.
Define the dimension functions $\frak{D}_{d}$, $\frak{D}^0_{d} : \cS_{d} \to [0,2]$ as in \eqref{DefDimFct}.

\begin{remark}
The asymptotic spectrum%
\footnote{ That is the spectrum of logarithmic approximation constants $-\log(c_d^+(z))$, $z\in \C$, where we use the 'liminf' in \eqref{DefApproxConstant}. }
$\cS^+_d$ is bounded below, contains a Hall ray and $\cS^+_d\cap \R$ 
 equals the closure of the logarithmic approximation constants of $z \in \C - \Q(i\sqrt{d})$ quadratic over  $\Q(i \sqrt{d})$ 
 (see \cite{Maucourant, PPPrescribing, VulakhDABianchi}).
 \end{remark}
 
Let $\cI_{d}$ be the ideal class group  of $\Q(i \sqrt{d})$ which contains only one ideal class if
 \be
\label{ChooseD}
	d \in \{1,2,3,7,11,19,43,67,163\}.
\ee

\begin{theorem}
Let $d$ be as in \eqref{ChooseD}.
The sets $\textbf{Bad}_{d}$ in $\C$ 
as well as $\textbf{Bad}_{d} \cap S(p/q,r)$ in $S(p/q,r)$ are absolute winning, where $S(p/q,r) = \partial B(p/q, r)$ with $(p,q) \in \cO_{d}^2$, $q \neq 0$ and $r \leq \tfrac{1}{2 \lvert q\rvert^2}$.
Moreover, there exist $t_0$ and $t_1$ such that $[t_0, \infty) \subset \cS_{d}$ and positive constants $k_0, k_u>0$  such that for $ t\geq t_0$ we have 
\be
\nonumber
	1 -  \frac{ k_0  }{ t \cdot e^{t/2} } \leq \frak{D}_{d}^0(t) \leq  2 - \frac{k_u}{t \cdot e^{4 t}},
\ee
as well as for $t_1 \leq t \in \cS_{d}$ 
\be
\nonumber
	2 -  \frac{ k_0  }{ t \cdot e^{t} } \leq \frak{D}_{d}(t) \leq  2 - \frac{k_u}{t \cdot e^{4 t}}.
\ee
\end{theorem}

\begin{proof}
The group $G=PSL(2, \C)$ acts on the real hyperbolic upper half space $\H^3$ (as a subset of $\C^2$) as the full group of orientation preserving isometries
and restricted to $\partial_{\infty}\H^3 = \C \cup \{ \infty\}$ by M\"obius transformations.
Moreover, the \emph{Bianchi group} $\Gamma_{d} \equiv PSL(2, \cO_{d})$  is a lattice in $G$, possibly with torsion, so that 
 $M_{d}= \H^3/\Gamma_{d}$ is a finite volume hyperbolic orbifold (note however that the previous results are still valid in this case).
The finitely many cusps in $M_{d}$ are in bijective correspondence to the ideal classes in $\cI_{d}$.
Thus, by assumption, let  $e$ be the only cusp in $M_{d}$ and consider the setting of Section \ref{SectionHeight}.
For the map $\C \ni z \mapsto v_z \in SH_0^+$ given in \eqref{CanonicalMap} we obtain the following correspondence.

\begin{lemma}
\label{BianchiDynamical}
We have $ \cH(v_z) = -\log(k \cdot c_{d}(z))$, where $k>0$ depends on the height function $\b_e$.
\end{lemma}

\noindent The proof of Lemma \ref{BianchiDynamical} follows in a similar way to the one in \cite{PPClosed} (see also \cite{Maucourant, VulakhDABianchi} again), using Lemma \ref{Dynamical} below and that $PSL(2, \cO_{d}) \cdot \infty = \Q(i \sqrt{d}) \cup \{\infty\}$.
Applying the results from Section \ref{SectionHeight} finishes the proof.
\end{proof}

\begin{remark}
The above theorem holds without the restriction that $d$ is as in \eqref{ChooseD}, that is, when $M_{d}$ has several cusps.
This follows along the lines of the respective proofs in \cite{Weil2, Weil3} and Section \ref{SectionBoundsDim} below,
where for \eqref{JarnikHeight} we need to replace bounded with respect to one cusp by \emph{uniformly bounded} in \cite{Weil3}, that is bounded with respect to all cusps with the same height.
\end{remark}


\subsubsection{Quadratic irrational numbers}
This section closely follows \cite{HPSpiraling, PPSpiraling} to which we also refer for further details.
Let either $K=\Q \subset  \R =\hat K$ or $K= \Q(i \sqrt{ d }) \subset \C = \hat K$, where $d \in \N$ is square-free.
Denote by $\cal{O}_K$ the ring of integers of $K$, that is $\Z$ or $\cal{O}_d$, and by $K_{quad}$ the real quadratic irrational (complex) numbers over $K$ in $\hat K$.
For $\b \in K_{quad}$, let $\b^{\sigma} \in K_{quad}$ be its Galois conjugate.
The subgroup $PSL(2, \cal{O}_K)$ of $PSL(2, \hat K)$, acting by M\"obius transformations on $\hat K \cup \{\infty\}$, 
preserves $K$ as well as $K_{quad}$ and $\psi(\b^{\sigma})= (\psi(\b))^\sigma$ for all $\b \in K_{quad}$, $\psi \in PSL(2, \cal{O}_K)$.

Fix  $\b_0 \in K_{quad}$ and let $\cal{P}_{\b_0} = PSL(2, \cal{O}_K) \cdot \{\b_0, \b_0^\sigma\}$ be its orbit in $\hat K$
which is dense  and countable.
For $x \in \hat K$ define its \emph{approximation constant} by
\be
\nonumber
	c_{\b_0}(x) \equiv \inf_{(\b,\b^\sigma)\in \cal{P}_{\b_0}} \lvert \b - \b^\sigma \rvert^{-1} \lvert x - \b \rvert,
\ee
which determines the set  \textbf{Bad}$_{\hat K, \b_0} \equiv \{ x \in \hat K: c_{\b_0}(x)>0\}$  and the spectrum $\cal{S}_{\hat K, \b_0} \equiv \{ - \log(c_{\b_0}(x)) : x \in \textbf{Bad}_{\hat K, \b_0}\}$;
for properties of the asymptotic spectrum $\cal{S}_{\hat K, \b_0}^+$  see \cite{PPSpiraling}.
Define the dimension function $\frak{D}_{\hat K, \b_0}: \cS_{\hat K, \b_0} \to [0,1]$ as in \eqref{DefDimFct}.

\begin{theorem} \textbf{Bad}$_{\hat K, \b_0}$ is absolute winning in $\hat K$. 
Moreover, there exist $t_0$ and $t_1$ such that $[t_0, \infty) \subset \cS_{\C, \b_0}$ and positive constants $k_0, k_u>0$  such that for $ t\geq t_0$ we have
\be
\nonumber
	 \frak{D}_{\C,\b_0}^0(t) \geq 1 -  \frac{ k_0  }{ t },
\ee
as well as for $t_1 \leq t \in \cS_{\hat K, \b_0}$ 
\be
\nonumber
	\frak{D}_{\hat K, \b_0}(t) \geq \text{dim}(\hat K) -  \frac{ k_0  }{ t \cdot e^{\text{dim}(\hat K) t/2} }.
\ee

\end{theorem}

\begin{proof}[Sketch of the proof]
For details of the following we refer to \cite{PPSpiraling}.
Note that $\cal{P}_{\b_0}$ determines a unique closed geodesic $\a_0$ in the modular surface in $M=\H^2/PSL(2, \Z)$, 
respectively in the Bianchi orbifold $M=\H^3/PSL(2, \cal{O}_d)$.
Moreover, each pair $(\b, \b^\sigma) \in \cal{P}_{\b_0}$ corresponds to a unique lift of $\a_0$.
Fix one of the cusps $e$ in $M$.
Moreover, we replace the base point $o$ with a sufficiently small cusp neighborhood $H_0$ of $e$ disjoint to $\cal{N}_{\e_0}(\a_0)$.
Note that the results in Section \ref{SectionLength} are still valid if we replace $SM_o$ with $SH_0^+$ 
and restrict to a compact fundamental domain $F$ in $\hat K$ for the action of $\Gamma_\infty=$ Stab$_{PSL(2, \cal{O}_K)}(\infty)$
(there is a constant $c_0$ such that geodesic rays $\gamma_{\tilde o,x}$ starting in a given point $\tilde o$ 
and $\gamma_{\tilde H_0,x}$ starting orthogonally to a horoball $\tilde H_0$ based at $\infty$ in $\H^2$ or $\H^3$
 and ending at the same point in $x\in F$
project to geodesic rays in $M$ whose spiraling lengths differs at most by $c_0$).
As in Lemma \ref{BianchiDynamical} (using Lemma \ref{Dynamical} below),
for the map $F \ni x \mapsto v_x \in SH_0^+$ given in \eqref{CanonicalMap}
there is a constant $c_1$ such that  $\lvert \cL(v_x) + \log(c_{\b_0}(x)) \rvert \leq c_1$.
Remarking that the compactness in \eqref{JarnikLength} was only required for the upper bound on the Hausdorff-dimension (see \cite{Weil3}, Section $3.3$),
we finish the proof by applying the results of Section \ref{SectionLength}.
\end{proof}


\subsubsection{Further applications}
Considering concrete lattices of the real (or complex) hyperbolic space,
further arithmetic applications can be obtained 
from results concerning the dynamics of the geodesic flow on the corresponding orbifold.
For  instance, for results of approximation of real Hamiltonian quaternions or of elements of a real Heisenberg group by 'rational elements' we refer to \cite{PPClosed, PPPrescribing} (avoiding a cusp),
and for their approximation by 'quadratic irrational elements' we refer to \cite{HPSpiraling, PPSpiraling} (avoiding a closed geodesic).


\subsubsection{Some discussion}

We conclude this section by further questions and remarks.
\begin{itemize}
\item[1.]
Can we determine additional properties of the above dimension functions $\frak{D}$ and $\frak{D}^0$?
For instance, 
the obtained results do not affect the question of whether or not  $\frak{D}$ and $\frak{D}^0$ are continuous on some subintervals of the spectra.
\\
In addition, it is not clear whether $\frak{D}$ and $\frak{D}^0$ are also positive outside of the determined intervals $[t_0, \infty)$.
Respectively, what is the value of $t_0$?
\item[2.]
The determined bounds for $\frak{D}$ and $\frak{D}^0$ are of an asymptotic flavor.
Moreover, they do not match and further effort for more precise bounds is necessary, in particular for $\frak{D}^0$.

\item[3.]
One may also study the dimension function $\frak{D}^1(t)$ defined as the Hausdorff-dimension of elements $x$ with approximation constant $t \leq c(x) < \infty$.
Moreover, as remarked earlier,  (for each of the above dimension functions) a lower bound for $\frak{D}$ gives a lower bound for $\frak{D}^+$, the asymptotic dimension function.
However, it seems to be hard to detemine an upper bound for $\frak{D}^+$.

\item[4.]
As remarked above, several of the properties of the asymptotic spectra  (height, distance and spiraling) rely on a result of \cite{Maucourant} for negatively curved manifolds.
The crucial tool in \cite{Maucourant} is Anosov's closing lemma.
Using a  'metric version' of the closing lemma in the context of proper geodesic CAT(-1) metric spaces (see \cite{WeilDiss})
we are able to show the denseness of approximation constants corresponding to periodic elements in the asymptotic spectrum (height, distance, spiraling)  in a more general setting.
This might have further applications, for instance to groups acting on metric trees. 
\end{itemize}


\section{Proofs}
\label{Proofs}
Recall that (in most cases) we restricted  to constant negative curvature and considered only finite volume hyperbolic manifolds $M=\H^{n+1}/\Gamma$.
The main reason for these restrictions is that, using Lemma \ref{Dynamical} below, 
we can relate our setup to a Diophantine setting on the full boundary $S^n=\R^n \cup \{\infty\}$ at infinity of $\H^{n+1}$
which provides the existence of suitable measures and can be partitioned in a nice way.
In particular we obtain the setup  and the requirements from our earlier work \cite{Weil3} and can apply its axiomatic approach in order to determine non-trivial bounds on the Hausdorff-dimension.

We begin in Section \ref{Preliminaries} with background and preliminaries and prove Theorem \ref{SpectrumProperties} in Section \ref{SectionSpectrumProperties}.
Section \ref{SecSetting} introduces the abstract setting and the abstract framework and proves necessary properties needed for the following sections.
Then in Section \ref{SectionBoundsDim} we prove the results about bounds on the Hausdorff-dimension and in Section \ref{SectionAbsWinningGame} we  prove the results about the absolute winning game.


\subsection{Some background and notation in hyperbolic geometry}
\label{Preliminaries}
A reference for further details and definitions of the following is given by \cite{Bridson, BGS}.
Let $\H^{n+1}$ be the $(n+1)$-dimensional real hyperbolic upper half-space model where $d$ denotes the hyperbolic distance on $\H^{n+1}$.
Assume all geodesic segments, rays or lines  to be parametrized by arc length and identify their images with their point sets in $\H^{n+1}$.
For a noncompact convex subset $Y \subset \H^{n+1}$, let $\partial_{\infty}Y$ denote its visual boundary, that is, the set of  equivalence classes of asymptotic rays in $Y$. 
Identify  $\partial_{\infty}\H^{n+1}$ with the set   $S^{n} \cong \R^{n} \cup \{\infty\}$ and equip  $\bar \H^{n+1} =\H^{n+1} \cup S^{n}$ with the cone topology. 
If $\gamma$ is a ray in $\H^{n+1}$ we will simply write $\gamma(\infty)$ for the corresponding point in $\partial_{\infty}\H^{n+1}$.
For any two points $p$ and $q$ in $\H^{n+1}$ denote  by $\gamma_{p,q}$ the geodesic segment, ray or line in $\H^{n+1}$ connecting $p$ and $q$.
Given $\xi \in \partial_{\infty} \H^{n+1}$ and $y\in \H^{n+1}$, the \emph{Busemann function} $\beta = \beta_{\xi,y} : \H^{n+1}  \to \R$ is defined by
\be
\nonumber
	\beta(x) \equiv \lim_{t \to \infty} d(y, \gamma_{y, \xi}(t)) - d(x, \gamma_{y, \xi}(t)),
\ee
which (exists and) is  continuous and convex on $\H^{n+1}$ and $\beta(y)=0$.
The sublevel sets  $H_t \equiv \b^{-1}([t, \infty))$ of $\beta$ are called \emph{horoballs} at $\xi$ (with respect to $y$).
If $\xi = \infty$, then $H_t$ equals $ \R^{n}\times [s, \infty)$ for some $s>0$,
and if $\xi \in \R^{n}$, then $H_t$ equals an Euclidean ball based at $\xi$. 
Given a horoball $C \subset \H^{n+1}$ based at $\partial_\infty C$ we can associate a Busemann function, denoted by $\b_C$ and parametrized such that $\b_C^{-1}([0, \infty)) = C$.
For three points $o, x, y \in \bar \H^{n+1}$, let $(x,y)_o$ denote the \emph{Gromov-product} at $o$ and for $\xi, \eta \in \partial_{\infty}\H^{n+1}$, let 
\be
\nonumber
	(\xi, \eta)_o \equiv \lim_{t\to \infty} (\gamma_{o, \xi}(t), \gamma_{o, \eta}(t))_o
\ee
be the extended Gromov-product at $o$.
Define the \emph{visual metric} at $o \in \H^{n+1}$ by $d_o:   \partial_{\infty}\H^{n+1} \times \partial_{\infty}\H^{n+1} \to [0, \infty)$ by  $d_o(\xi, \xi )\equiv 0$ and  for $\xi \neq \eta$ by 
\be
\nonumber
	d_o(\xi, \eta) \equiv e^{-(\xi, \eta)_o}.
\ee
Then $(\partial_{\infty}\H^{n+1}, d_o)$ is a compact metric space.
Note that the visual metric at a point $o\in \H^{n+1}$ is (locally) bi-Lipschitz equivalent  to the Euclidean metric $d_E$ on $\R^{n}$:
for every compact subset $K \subset \R^n$, there exists a constant $c_K>0$ 
such that for all $\xi, \eta \in K$,
\be
\label{BiLipschitz}
	c_K^{-1} d_o (\xi, \eta) \leq d_{E}(\xi, \eta) \leq c_K d_o(\xi, \eta);
\ee
see \cite{HPSpiraling}, Lemma 2.3.

Now let $M$ be a $(n+1)$-dimensional complete hyperbolic manifold.
Then there is a discrete, torsion-free subgroup $\Gamma$ of the isometry group of $\H^{n+1}$ identified with the (free) fundamental group $\pi_1(M)$ of $M$ acting on $\H^{n+1}$ such that the manifold $ \H^n/ \Gamma $ with the induced smooth and metric structure is isometric to $M$. 
Let $\bar \pi : \H^{n+1} \to  \H^{n+1}/\Gamma \cong M$ be the projection or covering map. 
When $\Gamma$ is a non-elementary geometrically finite discrete group, then we call $M$ \emph{geometrically finite}; see \cite{Bowditch} for background and definitions.

A sufficiently small cusp neighborhood of $M$ lifts to a  $\Gamma$-invariant countable collection $\cC$ of precisely invariant horoballs in $\H^{n+1}$.
In particular,  these horoballs are pairwise disjoint.
Note that by $\Gamma$-invariance we have $\gamma \circ \beta_C = \beta_{\gamma(C)}$ for every $\gamma \in \Gamma$ and $C \in\cC$;
in particular, each $\b_C$ projects to $\b_e$ on $M$ as in the introduction.

A closed geodesic $\a$ in $M$ lifts to a  $\Gamma$-invariant countable collection $\cC$ of geodesic lines in $\H^{n+1}$ 
which is \emph{$(\e, T)$-immersed} (using the terminology of \cite{HPSpiraling,PPPrescribing}), that is, given $\e>0$ there exists a $T=T(\e)>0$ 
such that for any two distinct lines $C_1$ and $C_2$ in $\cC$ we have that the diameter 
\be
\nonumber
	\text{diam}(\cN_\e(C_1) \cap \cN_\e(C_2)) \leq T; 
\ee
here and hereafter, $\cN_\e(S)$ denotes the closed $\e$-neighborhood of a set $S$ in a metric space.

A point $x$ in $M$ lifts to a $\Gamma$-invariant countable collection $\cC$ of points which are \emph{$\tau_0$-separated} and, if $M$ is compact, \emph{$R_0$-spanning} for some $\tau_0>0$ and $R_0>0$;
that is, for any distinct points $z, y \in \cC$ we have $d(z,y) \geq \tau_0$  and for any point $z \in \H^{n}$ there is $y \in \cC^3$ such that $d(z,y) \leq R_0$.

Note that for each of these collections, given a point $o \in \H^{n}$, the set $\{d(o, C) :  C \in \cC\} \subset \R$ is discrete and unbounded.
Finally, note that the dynamics of a geodesic ray in $M$, in terms of penetration properties of a cusp neighborhood or a  neighborhood of a  point and a  closed geodesic,
corresponds to the dynamics of its lift to $\H^n$ in the collection $\cC$ as above.


\subsection{Proof of Theorem \ref{SpectrumProperties} (Properties of the asymptotic distance spectrum)}
\label{SectionSpectrumProperties}

In order to avoid giving further definitions and background, we only sketch the proof and refer to \cite{Bowditch} for proper definitions and further details.
Recall that $M$ is a complete  geometrically  finite Riemanninan manifold of curvature at most $-1$.
Let $\cC M$ denote the \emph{convex core} of $M$ which is a closed convex subset of $M$ that can be decomposed into a compact subset $K$ and, 
unless $M$ is convex-cocompact, into a disjoint union of  open sets $V_i$ 
(where $V_i = \pi(\tilde V_i) \cap \cC M$ and each $\tilde V_i$ consists of a countable collection of disjoint horoballs in the universal cover of $M$). 
Let $D_0$ be the diameter of $K$, $x \in K$,  and $\gamma$ be a geodesic ray (not necessarily starting in $o$).
If $\gamma$ is  positively recurrent, then it is eventually contained in $\cN_{D_0}(\cC M)$ and it follows from the decomposition that it must intersect the compact set $ B(x, 2 D_0) \supset K$ infinitely often. 
This implies that the asymptotic distance spectrum is bounded below by 
\be
\nonumber
	-\log(d(x_0, x) + 2D_0).
\ee
Recall that $\pi : SM \to M$ denotes the footpoint projection, set $d_{x_0} (z) \equiv d(x_0, z)$ and define 
\be
\nonumber
	f \equiv (-\log) \circ d_{x_0} \circ \pi : SM \to \R.
\ee
Cleary, both $\pi$ and $d_{x_0}$ are continuous and proper functions so that the same is true for $f$.
Let $\cP \subset SM$ denote the set of unit  vectors tangent to closed geodesics in $M$.
It follows from \cite{Maucourant}, Theorem $2$, that
\be
\label{Closure}
	\R \cap \{ \limsup_{t \to \infty} f(\phi^t(v)) : v \in SM\} =  \overline{ \{ \max_{t \in \R} f(\phi^t(w)) : w \in \cP \} }.
\ee
Note that $\limsup_{t \to \infty} f(\phi^t(v))$ depends only on the asymptotic class of $\gamma_v$, hence not on the base point,
and we may in fact replace $SM$ in the left hand side of \eqref{Closure} with $SM_o$.
Finally, by definition $\limsup_{t \to \infty} f(\phi^t(v) = \cD^+(v)$ as well as 
\be
\nonumber
	\max_{t \in \R} f(\phi^t(w)) = - \big( \min_{t \in \R} \log( d(x_0, \alpha(t))) \big) \equiv - \big( \log(d(x_0, \a))\big),
\ee
where $\a$ denotes the closed geodesic determined by $w$.
This finishes the proof.


\subsection{ The setting and the abstract framework}
\label{SecSetting}

\subsubsection{The setting}
\label{SectionSetting}
We first introduce our setting which, as remarked in the previous Section \ref{Preliminaries}, 
is for instance satisfied when lifting our setup from Section \ref{GeodesicFlow} (as stated above). 
Consider three nonempty countable collections $\cC^i$, $i=1,2,3$, of closed convex sets in $\H^{n+1}$, where
\begin{itemize}
\item[1.]  $\cC^1$ is a collection of pairwise disjoint horoballs
\item[2.] 	$\cC^2$ is a collection of  $(\e_0, T)$-immersed geodesic lines,
\item[3.]	$\cC^3$ is a collection of points which is $\tau_0$-separated and $R_0$-spanning.
\end{itemize}

\begin{remark}
Note that if $\cC^2$ is $(\e_0, T)$-immersed then it is also $(\d, L)$-immersed for $\d >0$ and $L=L(T, \e_0, \d)$.
Concerning the bounds on the Hausdorff-dimension let us remark the following.
Using a result of \cite{StratmannUrbanski} about the Patterson-Sullivan measure of a non-elementary convex-cocompact Kleinian group and similarly to \cite{Weil3},
the collection $\cC^2$ may be replaced by a $(\e_0, T)$-immersed collection of totally geodesic up to $(n-2)$-dimensional submanifolds of $\H^{n+1}$.
Moreover, the assumption that $\cC^3$ is $\tau_0$-separated is only needed for the following lower bounds on the Hausdorff-dimension 
and the upper bound uses that it is $R_0$-spanning.
\end{remark}

\noindent
Fix a base point $o\in \H^{n+1} \cup S^n$.
In the first case, we fix a horoball  $C^1_o\in \cC^1$ (which we assume to be) based at $\infty \in S^n$ of Euclidean height $1$ 
and exclude it from the collection. 
Then the Hamenst\"adt-metric $ d_{C^1_o}$ on $\R^{n}=\partial_{\infty}\H^{n+1} - \{\infty\} \equiv \bar X_1$ with respect to $C^1_o$ equals the Euclidean metric;
by abuse of notation, we write $d_o$ for $ d_{C^1_o}$ and $\gamma_{o, \xi}$ for a vertical ray starting on $C_o^1$ and ending at $\xi \in \R^n$.
For the second and third case we let $o = C^i_o \in \H^{n+1}$ such that $o \not \in \cup_{C \in \cC^2} \cN_{\e_0}(C)$ 
or $o \not \in \cup_{C \in \cC^3} B_{\e_0}(C)$ for some $\e_0>0$ respectively. 
Let $d_o$ be the visual metric at $o$ on the sphere $S^{n-1}\equiv \bar X_i$.

Each collection determines a  set of \emph{sizes} 
\be
\nonumber
	\{ s^i_C \equiv d(C^i_o, C) : C \in \cC^i \} \subset \R_{\geq 0} 
\ee
which we assume to be discrete and unbounded.
For a point $x \in \H^n$, distinct to $o$, we let $\gamma_{o,x}(\infty)\in \partial_\infty\H^{n+1} =S^n$ be the boundary projection of $x$ with respect to $o$;
by abuse of notation, we write $\partial_\infty x \equiv \gamma_{o,x}(\infty)$ in the following.
Then, each collection determines a  nonempty collection $\cC^i_\infty$ at infinity, where
\be
\nonumber
	 \cal{C}^i_{\infty}\equiv \{\partial_{\infty}C \subset \bar X_i : C  \in \cal{C}^i  \} \subset \partial_\infty \H^{n+1 }= S^{n} 
\ee
is the collection of tangency points of the horoballs in $\cC^1$, endpoints of the geodesic lines in $\cC^2$ or  boundary projections of points in $\cC^3$ with respect to $o$, respectively.

The following Lemma and Proposition are crucial.
The Lemma relates the 'Diophantine' properties of a point $\xi \in \bar X_i$ with respect to the collection $\cC^i_\infty$ at infinity
with the (dynamical) penetration properties of the ray $\gamma_{o,\xi}$ in the collection $\cC^i$.

\begin{lemma}
\label{Dynamical}
There are universal constants $\k_u\geq \k_l >0$ and $\bar c_0\geq 0$ with the following property.
Let $\gamma=\gamma_{o, \xi}$ be a geodesic line (or ray) starting from $o$ and let $C \in \cC^i$ 
with $\b_C (\gamma) \geq 0$, $L(\gamma \cap \cal{N}_{\e_0}(C))\geq c_0$ or, for the the third case that $d(o, C) \geq c_0$ and $d(\gamma, C) \leq e^{-c_0}$.
Then
\begin{itemize}
	\item[1.] $d_o(\xi, \partial_\infty C) = \tfrac{1}{2}e^{-c} e^{-d(C^i_o, C)}$ if and only if the height $\max_{t\in \R} \b_C(\gamma(t)) =c $; 
	\item[2.] $ \k_l  \cdot  e^{-L(\gamma \cap \cal{N}_{\e_0}(C))} \cdot e^{-d(o,C)}  \leq d_o(\xi, \partial_{\infty}C)  \leq \k_u  \cdot e^{-L(\gamma \cap \cal{N}_{\e_0}(C))} \cdot  e^{-d(o,C)} $;
	\item[3.] $ \k_l \cdot d(\gamma_{o, \xi}, C) \cdot e^{-d(o,C)}  \leq  d_o(\xi, \partial_{\infty}C) \leq  \k_u \cdot  d(\gamma_{o, \xi}, C)\cdot e^{-d(o,C)} $.
\end{itemize}
\end{lemma}

\begin{proof}
The first part follows from elementary hyperbolic geometry.
The second and third part follow along the lines of the proof of  Lemma $3.11$ in \cite{Weil2}  (see also \cite{HPDiophantine,HPSpiraling}), 
which is stated in a slightly different way.
\end{proof}

Moreover, we need the following result on the distribution of the sets in the collections $\cal{C}_{\infty}^i$ in $\partial_{\infty}\H^{n+1}$.

\begin{proposition}
\label{Distribution}
Let  $l_1\equiv  -\log(2)$ and $ l_2 \geq T + \e_0$ be sufficiently large.
Then, for the Cases 1, 2, we have for distinct sets $C, \bar C\in \cC^i$ that
\be
	d_o(\partial_{\infty}C, \partial_{\infty}\bar C) \geq e^{-l_i} \cdot e^{- \max\{s_C^i, s_{\bar C}^i \}}.
\ee
Moreover, there exists a constant $k_0=k_0(\tau_0)$ such that, for every $c>0$ and every ball $B=B_{d_o}(\eta, 2 \cdot e^{-t})$ with $t\geq 0$,
\be
	\lvert \{ \partial_\infty x \in B : x \in \cC^3 \text{ with } s_x=d(o,x)  \in (t-c, t]\}\rvert \leq k_0 \cdot c.
\ee
On the other hand, there exist a time $t_3\geq 0$ and a constant $u_* = u_*(R_0)\geq 0$  such that for every ball $B=B_{d_o}(\eta, e^{-(t-u_*)})$ with $t\geq t_3$ sufficiently large
there exists $x \in \cC^3$ with $t-u_* \leq d(o,x) \leq t$ and $\partial_\infty x \in B$.
\end{proposition}

\begin{proof}
For the first case, recall that each $C \in \cC^1$ is a Euclidean ball in $\R^n\times \R^+$, tangent to a point $\eta \in\R^n$ and of Euclidean radius $r_\eta = e^{-d(C_o^1, C)} /2$.
Hence, consider two such disjoint Euclidean balls, tangent to two distinct points $\eta \neq \bar \eta$ in $\R^n$ and with Euclidean radius $r_\eta$, $r_{\bar \eta}$.
A simple Euclidean computation (Pythagoras) shows that 
\be
\nonumber
	d_{C_o^1}(\eta, \bar \eta) = \lvert \eta - \bar \eta \rvert \geq 2 \sqrt{r_{\eta}r_{\bar \eta}} \geq 2 \min\{r_{\eta}, r_{\bar \eta} \} = e^{-l_i}  \cdot e^{- \max\{d_i, d_j\}}.
\ee
The second case follows from Propostion 3.12 in \cite{Weil2}.

The same is true for the first part of the third case.
Since we will make use of this part twice, we recall it for the sake of completeness.
For a subset $Y\subset S^n$ and $0\leq a \leq\bar a$, consider the truncated cone of $Y$ with respect to $o$,
\be
\nonumber
	Y(a,\bar a) \equiv \{ \gamma_{o, \xi} (t) \in \H^{n+1} : \xi \in Y, a\leq t \leq \bar a \}.
\ee
Fix $c>0$, a ball $Y=B_{d_o}(\eta, 2e^{-t})$  and 
note that a point $x_{\infty}$ with $t-c< d(o, x) \leq t$ lies in $Y$ if and only if  $x \in Y(t-c, t)$.
It therefore suffices to estimate the cardinality of $Y(t-c,t) \cap \cC^3$.

First, we claim  that $Y(t-c, t)$ is contained in
the $(\delta_0 + 2\log(2))$-neighborhood of the geodesic segment $\gamma_{o, \eta}((t-c,t])$, where $\d_0$ denotes the hyperbolicity constant of $\H^n$.
To see this, note that for the Gromov-product for  $\xi \in Y$ and $\eta$ at $o$
\be
\nonumber 
	(\xi, \eta)_o \geq - \log(d_o(\xi, \eta)) \geq t - \log(2)
\ee
and hence, see \cite{Bridson}, $d(\gamma_{o, \xi}(s), \gamma_{o, \eta}(s)) \leq \d_0$, for all $s \leq t - \log(2)$.
For $t-\log(2) \leq s \leq t$ we have 
\bea
\nonumber
	d(\gamma_{o, \xi}(s), \gamma_{o, \eta}(s)) &\leq& d(\gamma_{o, \xi}(s), \gamma_{o, \xi}(t-\log(2))) + \delta_0 + d(\gamma_{o, \eta}(s), \gamma_{o, \eta}(t-\log(2))) 
	\\ \nonumber
	&\leq& \delta_0 + 2\log(2),
\eea
concluding the claim.

Clearly, since $\H^{n+1}$ is of constant sectional curvature, 
there exists a universal constant $k>0$ such that the hyperbolic volume of $\cal{N}_{\d_0 + 2\log(2)}(\gamma_{o, \eta}((t-c,t]))$ is bounded by $k\cdot c$.
Since moreover $\cal{C}_3$ is $\tau_0$-separated 
it also follows that there exists a constant $\bar k= \bar k(\tau_0)>0$ such that the  (hyperbolic) volume of every ball $B(x, \tau_0/2)$ is at least $\bar k$.
Thus, we conclude that $| Y(t-c,t) \cap \cC^3 | \leq k/\bar k \cdot c$, as stated above.

For the remaining part, since $\cC^3$ is $R_0$-spanning,
consider an element $x\in \cC^3$ such that $d(\gamma_{o, \eta}(t-R_0), x) \leq R_0 $.
Hence, 
\be
\nonumber
	t- 2 R_0 \leq d(o,x)\leq t.
\ee
Moreover, when $t_3$ is sufficiently large with respect to $R_0$ and a given $\e>0$, 
then it follows from hyperbolic geometry that for some constant $\tilde u_*=\tilde u_*(\e, R_0)$ 
we have $d(\gamma_{o, x}(s), \gamma_{o, \eta}(s))\leq \e$ for all $s\leq t - R_0 - \tilde u_*$.
Thus, setting $u_* = R_0 + \tilde u_*$ for a suitable $\e>0$, it follows from Lemma \ref{Dynamical} that
$\partial_\infty x = \gamma_{o, x}(\infty) \in B_{d_o}(\eta, e^{-(t-u_*)})$.

This finishes the proof.
\end{proof}

\subsubsection{The abstract framework}
\label{SectionAbstractFramework}

We now introduce the framework of \cite{Weil3}, which is slightly different and already adopted to our setting. 
When there is no need to distinguish between the particular cases, we will omit the index $i$ in the following and consider with $\cC$ a collection as above.
For a closed subset $X \subset \bar X$ of the proper metric space $\bar X$, we define the set of \emph{badly approximable points} in $X$ (with respect to the collection $\cC$) 
\bea
	\nonumber
	\textbf{Bad}_{X}(\cC) &\equiv& \{\xi \in X : \exists c=c(\xi)>0 \text{ such that } d_o(\xi,  \partial_{\infty}C) \geq c \cdot e^{-d(C_o, C)} \text{ for all } C \in \cC \}
\eea
and the subset of badly approximable points with approximation constant at least $e^{-c}$, $c< \infty$, which is given by
\bea
	\nonumber
	\textbf{Bad}_{X}(\cC, c) &\equiv& \{\xi \in X : d_o(\xi,  \partial_{\infty}C) \geq e^{-c} \cdot e^{-d(C_o, C)} \text{ for all } C \in \cC \}.
\eea

Define the one-parameter family $\cR$ consisting of the \emph{resonant sets} $\cR(t)$ of size at most $s_t \equiv t  \geq 0$ by
\be
\nonumber
	 \cR(t)  \equiv \{ \partial_\infty C : C  \in \cC \text{ such that }  s_C = d(C_o, C) \leq t \} \subset \bar X.
\ee
For $c>0$ define also $\cR(s, c)$ which denotes the resonant sets with sizes in the window $(s-c, s]$, that is
\be
\nonumber
	 \cR(s, c) \equiv \cR(s) -  \cR(s- c)  = \{ \partial_\infty C \in \cC_\infty:  s - c \leq d(C_o, C) \leq s \}.
\ee

Given a subset $X$ and a technical parameter $t_* \geq 0$, needed below, we determine the parameter space $(\Omega, \psi)$ as follows.
Define (for the respective metrics) the monotonic%
\footnote{ A set-valued function $\psi$ on $\Omega$ is monotonic if $\psi(x,t +s ) \subset \psi(x, t)$ for all $(x, t) \in \Omega$ and $s\geq 0$. }
 function $\psi$ on  the set of \emph{formal balls} $\Omega \equiv  X \times [t_{*}, \infty)$ by
\be
\nonumber
	\psi( \xi, t) \equiv B_{d_o}(\xi, e^{-t}) \cap X,  \ \ \ \ (\xi,t) \in \Omega,
\ee
which is the restriction of the monotonic function $\bar \psi(x,t) \equiv B_{d_o}(x, e^{-t}) \subset \bar X$, $(x, t) \in \bar X \times \R^+$ to $\Omega$.
Denote by $\cN_{d_o}(\cR(t), r)$  the closed $r$-neighborhood of the set $\cR(t)$ in $\bar X$ with respect to the metric $d_o$.

\begin{remark}
Note that we have diam$(\psi(x,t)) \leq 2 e^{- t}$, which corresponds to the condition $[\sigma]$ for $\sigma =1$ in \cite{Weil3}.
Moreover, setting $d_* \equiv \log(3)$, we remark that, since the resonant set $\cR(t)$ is discrete for all $t\geq t_*$, it follows for all $\xi \in X$ that 
\bea
\nonumber
	&& \xi \not \in \cN(\cR(t), e^{-t}  ) \implies B_{d_o}(\xi,e^{-(t+d_*)}) \cap \cN_{d_o}(\cR(t), e^{-(t +d_*)} )= \emptyset;
\eea
hence condition $[d_*, \cF]$ (as well as $[d_*]$) of \cite{Weil3} is satisfied.
In particular this holds for $\psi(\xi, t+d_*)$ replaced by $B_{d_o}(\xi,e^{-(t+d_*)})$.
\end{remark}

\subsection{Bounds on the Hausdorff-dimension}
\label{SectionBoundsDim}

Recall  the setting and abstract framework introduced in the previous Section \ref{SecSetting}:
in the following assume we are given a collection $\cC$ and the parameter space $(\Omega, \psi)$ as above.
We start by deducing a lower bound on the Hausdorff-dimension for \textbf{Bad}$_X(\cal{C}, c)$ under abstract conditions in Section \ref{SettingAbstractLB}.
After that we verify these conditions and apply this lower bound in Sections \ref{SectionHallRay} and \ref{SectionJTI}.

\subsubsection{The (abstract) lower bound}
\label{SettingAbstractLB}
Let $\cC$ be a collection and $(\Omega, \psi)$ be a parameter space  as in Section \ref{SecSetting}.
Fix  $l_*\in \R$, $d_*=\log(3)$ and consider the following conditions, given the parameter $c>0$.
\begin{itemize}
\item[(S0)] There exists a formal ball $\omega_0 \equiv (x, t_*) \in \Omega$ such that
	\be
	\nonumber
		\psi(\omega_0) \subset X -  \bigcup_{ C \in \cC^i:  \ s_C \leq t_* - l_*-  c} \cN_{d_o}( \partial_\infty C, e^{-(s_C + 2c + l_*)} ) .
	\ee
\end{itemize}

Let $\mu$ be a locally finite Borel measure on $\bar X$.
\begin{itemize}
\item[($\mu1$)] 
$(\Omega, \psi, \mu)$ satisfies a \emph{power law with respect to the parameters $(\tau, c_l, c_u)$},
where  $\tau>0$ and $c_{u}\geq c_{l}>0$,
that is, we have  supp$(\mu)=X$ and
\be
\nonumber
	c_l e^{- \tau  t } \leq \mu( \psi(x,t) ) \leq c_u e^{- \tau t}
\ee
for all formal balls $(x,t) \in \Omega$.	
\item[($\mu2$)] 
\label{mu2}
 $(\Omega, \psi, \mu)$ is called  \emph{$\tau_l(c)$-decaying with respect to $\cal{R}$}, 
if all formal balls $\omega=(\xi, t + d_*) \in \Omega$ we have
\be
\label{Decaying}
	\mu(\psi(\omega) \cap \cN_{d_o}( \cR(t- l_*, c), e^{-(t + c - d_*)} ) \leq \tau_l(c) \cdot \mu(\psi(\omega)),
\ee
where $\tau_l(c)<1$ is a constant depending on $c$.
\end{itemize}

\begin{remark}
Condition $(S0)$ is trivially satisfied for all $(\xi, t_*) \in \Omega$ whenever $c\geq t_*- \min\{ s_C : C \in \cC\} $.
Note that condition $(\mu1)$ reflects how well a ball in $X$ can be separated into smaller balls of the same radius and could be stated in different terms.
\end{remark}

These conditions imply the following lower bound.

\begin{proposition}
\label{LBFormula}
Under these conditions and in our setting we have
\bea
\nonumber
	 \emph{\text{dim}(\textbf{Bad}}_X(\cal{C}, 2c  + l_*))
	 &\geq &
	\tau
	- \frac{  \log( 2 c_u^2  c_l^{-2}e^{2\tau d_*}) + \lvert \log(1-\tau_l(c)) \rvert  }
	{ c }.
\eea
\end{proposition}

\begin{remark} In the case that $X=\R^n$, a more precise lower bound can be determined; see below for the Jarn\'ik-type inequality.
Moreoever, under a condition converse to \eqref{Decaying} a similar upper bound is given in \cite{Weil3}.
\end{remark}

\begin{proof}  
The proof follows from the application of the axiomatic approach of \cite{Weil3}, Section $2.2$.
More precisely, consider the 'restricted family' $\cC^*$  defined by $\cC^* \equiv \{ C \in \cC : s_C \geq t_* - c - l_*\}$ given the technical parameter $t_* $.
Using the Conditions $(\mu1)$, $(\mu2)$ and $[d_*]$, $[d_*, \cF]$, Proposition $2.10$ of \cite{Weil3} establishes Condition $[\tau(c)]$ of \cite{Weil3} with the parameter
\be
\nonumber
	\tau(c) = \frac{(1-\tau_l(c)) \tfrac{c_l}{c_u} e^{- \tau c}}{2 \tfrac{c_u}{c_l} e^{-\tau (c - 2d_*) } }.
\ee
From Theorem 2.4 in \cite{Weil3} (using $[\sigma]$ with $\sigma=1$ and $(\mu1)$ again) we have
\be	
\nonumber
	\text{dim}(\textbf{Bad}_{\psi(\omega_0)}(\cC^*, 2c + l_*)) \geq 	\tau 	- \frac{  | \log(\tau(c) |}{ c }.
\ee
Finally, by $(S_0)$ we see that $\textbf{Bad}_{\psi(\omega_0)}(\cC^*, 2c + l_*) \subset \textbf{Bad}_{X}(\cC, 2c + l_*)$, finishing the proof.
\end{proof}

 
\subsubsection{Proof of the  Theorems \ref{ThmHeight}, \ref{ThmPoint}, \ref{ThmLength} (Hall ray-type results)}
\label{SectionHallRay}
Given $C \in \cC^i$ and $o \in \H^{n+1} \cup S^n$ as above, we define the following three maps $\frak{p}_{C_o^i, C}^i: \bar X_i \to [0, \infty]$, where
\begin{itemize}
\item[1.] $\frak{p}_{C_o^1, C}^1(\xi) \equiv \cH_{C_0^1,C}(\xi)$, where $\cH_{C_o^1,C}(\xi) \equiv \sup_{t \in \R} \b_C(\gamma_{C_o^1, \xi}(t))$,
\item[2.]  $\frak{p}_{C_o^2, C}^2(\xi) \equiv \cL_{o,C}(\xi)$, where $\cL_{o,C}(\xi) \equiv L(\cN_{\e_0}(C) \cap \gamma_{o, \xi} )$,
\item[3.]  $\frak{p}_{C_o^3, C}^2(\xi) \equiv  \cD_{o,C}(\xi) $, where $\cD_{o,C}(\xi) \equiv - \log( d(C, \gamma_{o, \xi}))$. 
\end{itemize}

Let $t_i\geq 0$ be technical constants, where $t_1=t_2=0$ and $t_3 \geq \bar c_0$ 
for the constant $\bar c_0$ from Lemma \ref{Dynamical}.
Choose one of the convex sets $C_0 = C_0^i\in \cC^i$ with $s_0^i\equiv d(C_o^i, C_0)$ minimal under the condition $d(C_o^i, C_0) \geq t_i$ 
(which exists by discreteness).
For the third Case  assume $t_3=s_0^3$ and note that $t_3$ corresponds to the constant $t_0$ in Section \ref{Introduction}. 
Since we only consider geodesic rays starting from $o$ for times $t\geq t_3$ and by the choices below, 
we may simply ignore all points $x \in \cC^3$ with $d(o, x)< t_3$ in the following and delete them from the collection $\cC^3$.
This follows from the next remark.

\begin{remark}
Let $x \in \cC^3$ with $d(o, x)< t_3$, hence $x \neq C_0$.
For any time $t\geq t_3$ and geodesic ray $\gamma$ starting from $o$ and with  $d(\gamma, C_0)\leq e^{-c_0}$,
we have 
\be
\nonumber
	d(\gamma(t), x) \geq d(\gamma(t_3), x) \geq  d(C_0, x) - d(\gamma(t_3), C_0) \geq \tau_0 -e^{-c_0} \geq \tau_0/2
\ee
for $c_0 \geq -\log(\tau_0/2)$.
\end{remark}

For $c_0>0$ sufficiently large (with $c_0 \geq \bar c_0$), the main idea is  to define the set
\be
\nonumber
	X_i \equiv (\frak{p}_{o, C_0}^i)^{-1}(c_0),
\ee
in $\bar X_i$, which is diffeomorphic to a $(n-1)$-dimensional Euclidean sphere; see Lemma \ref{PowerLaw} below.
By choice, for each of the cases, a ray $\gamma_{C_o^i, \xi}\lvert_{[t_i, \infty)}$ with $\xi \in X_i$ will 'hit' first the set $C_0 \in \cC^i$ 
and has exactly the desired penetration property with respect to the parameter $c_0$.
In view of Lemma \ref{FinalStep} below, given a further large parameter $s_0$, our aim is to show the existence of a subset $A$ of $X_i$
for which any given $\xi \in A$ satisfies
\be
\nonumber
	d_o(\xi, \partial_\infty C) \geq \k_u \cdot e^{-s_0} e^{-d(C_o^i, C)}, 
\ee
for all $C_0 \neq C \in \cal{C}^i$  where $\k_u$ is from  Lemma \ref{Dynamical},
and with a lower bound on the Hausdorff-dimension of $A$ depending on the parameter $s_0$.

More precisely,
set $\bar k_u \equiv - \log(\k_u)$ and choose $ l_*^i \equiv  l_i + \log(3)$, where $l_i$ is given in Proposition \ref{Distribution} and $l_*^3= \log(2)$.
Given $s_0\geq  l_*^i - \bar \k_u$  with $s_0 \geq \bar c_0$
 let $c\geq 0$ such that $s_0 + \bar \k_u = 2c + l_*^i$, that is
\be
\nonumber
	c= \frac{s_0 + \bar \k_u - l_*^i}{2}.
\ee
Then, we exclude the set $C_0$ from the collection $\cC^i$ and choose $A\equiv \textbf{Bad}_{X_i}(\cal{C}^i, s_0 + \bar \k_u)$ 
and remark that, in fact, $A$ projects (locally injectively) to a  subset $S(c_0, s_0)$.

\begin{lemma}
\label{FinalStep}
Given $\xi \in \textbf{Bad}_{X_i}(\cal{F}, s_0 + \bar \k_u)$ we have  $\frak{p}_{o, C_0}^1(\xi) = c_0$ and  $\frak{p}_{o, C}^i(\xi)\leq s_0$ for all $C_0 \neq C \in \cC^i$
with $d(C_o^i, C)\geq t_i$;
in particular, $\gamma \lvert_{[t_i, \infty)}$ projects to a geodesic in $S(c_0, s_0)$.
\end{lemma}

\begin{proof}
By construction, every $\xi \in X_i$ satisfies $\frak{p}_{o, C_0}^i(\xi) = c_0$.
Let $C_0 \neq C \in \cC^i$ with $d(C_o^i, C)\geq t_i$.
It follows from the definition of $\textbf{Bad}_ {X_i}(\cal{C}, s_0 + \bar \k_u)$ that $d_o(\xi, \partial_\infty C) \geq \k_u \cdot e^{-( s_C^i + s_0)}$.
Thus, by Lemma \ref{Dynamical}, we have that 
\be
\nonumber
	\frak{p}_{o, C}^i(\xi)\leq - \log(\tfrac{1}{\k_u} e^{s_C^i} d_o(\xi, \partial_\infty C) ) \leq  s_0,
\ee
as claimed.
Recalling from the above remark that we may ignore all $x \in \cC^3$ with $d(o,x)< t_3$, the proof follows.
\end{proof}

Thus, a lower bound on the Hausdorff-dimension of $A$ will be a lower bound on the dimension of $S(c_0, s_0)$.
For the respective cases, set 
\be
\nonumber
	t_*^i=t_*^i(c_0) \equiv  s_0^i + c_0 +  \log( \bar c_n)+ \log(3) + \log(2), \ \ \ \Omega_i \equiv X_i \times [t_*^i, \infty),
\ee
where $\bar c_n \geq 1$ is determined in the proof of Lemma \ref{PowerLaw} and independent of $s_0^i$,  $c_0$ (and $s_0$).

In order to obtain a lower bound for dim$(A)$, we check conditions $(S0)$, $(\mu1)$ and $(\mu2)$.
Recall that condition $(S0)$ is trivially satisfied for all $\omega_0=(\xi, t_*^i) \in \Omega$ whenever $c\geq t_*^i- s_0^i$ (note again that $s_0^i= \min\{s_C^i : C \in \cC^i \}$ for all three cases); hence for 
\be
\nonumber
	s_0 \geq 2c_0 + 2\lvert \log(6 \bar c_n)\rvert + l_*^i - \bar k_u \equiv 2c_0 + k_1.
\ee
More generally, condition $(S0)$ is satisfied in the following situations.

\begin{lemma}
For Cases 1, 2, when $c$ is sufficiently large (independent on $c_0$), 
that is when
\be
	\nonumber
	e^{ s_0}e^{  - l_i  }(1- e^{-c } \cdot \k_u e^{ \log( 6\bar c_n)}) \geq \k_u,
\ee
then for all $C_0 \neq C \in \cC^i$ with $s_C^i \leq t_* - c - l_*^i$ we have
\be
	\nonumber
	d_o(X_i, \partial_\infty C) \geq  \k_u \cdot e^{- (s_C^i+s_0) };
\ee
in particular $(S0)$ is satisfied for any $\omega_0=(\xi, t_*^i) \in \Omega_i$.
\end{lemma}

\begin{proof}
Let $\xi \in X_i$ be any point.
Given $C_0 \neq C \in \cC^i$ with  $s_0^i  \leq s_C^i \leq t_* - l_*^i- c$,
using Proposition \ref{Distribution} and Lemma \ref{Dynamical}, we have
\bea
	\nonumber
	d_o(\xi,  \partial_\infty C) &\geq& d_o(\partial_\infty C_0,  \partial_\infty C) - d_o(\partial_\infty C_0, \xi)
	\\ \nonumber
	 &\geq& e^{-l_i} e^{-s_C^i} - \k_u e^{- ( s_0^i + c_0)}
	\\ \nonumber
	& \geq& e^{- (s_C^i +s_0) } \cdot e^{ s_0 } (e^{  - l_i  } - \k_u e^{  s_C^i  - ( s_0^i + c_0)}) 
	\\ \nonumber
	& \geq& e^{- (s_C^i +s_0) } \cdot e^{ s_0 } (e^{  - l_i  } - \k_u e^{  t_* - ( s_0^i + c_0 + c +  l_*^i )}) 
	\\ \nonumber
	& =& e^{- (s_C^i +s_0) } \cdot e^{ s_0}e^{  - l_i  }  (1- e^{-c } \cdot \k_u e^{  \log(\bar c_n) + \log(6)}) 
	\equiv e^{- (s_C^i+s_0) } \cdot h_*,
\eea
where $h_* = h_*(s_0) \geq \k_u$ is independent on $c_0$.
\end{proof}

We need to establish the following crucial result.

\begin{lemma}
\label{PowerLaw}
$X_1, X_3 $ are isometric to and $X_2$ is diffeomorphic to a $(n-1)$-dimensional Euclidean sphere.
Moreover, there exist measures $\mu_i$ such that $(\Omega_i, \psi_i, \mu_i)$ satisfies a power law with respect to the exponent $\tau = n-1$ 
and constants $c_u = \bar c_n \cdot e^{-(n-1)(s_0^i+c_0)}$, $c_l  = \bar c_n^{-1} \cdot e^{-(n-1)(s_0+s_1)}$ 
where $\bar c_n\geq 1$ is independent from $s_0^i$ and $c_0$; 
hence $(\mu1)$ is satisfied.
\end{lemma}

\begin{proof}
For the second case assume that $\partial_\infty C_0^2$ equals  $\{0, \infty\}$ 
and let $x_2$ denote the (unique) point on the vertical line $C_0^2$ at distance $d(o, C_0^2) = s_0^2$ to $o$.
For the third case assume  $o=e_{n+1}$ and $\partial_\infty C_0^3=0 \in \R^n$.
We may also assume that $x_2=e_{n+1}$ and in addition, for $c_0$ sufficiently large, that 
$X_2$ and $X_3$ are contained in the unit ball around $0\in \R^n$ on which $d_{e_{n+1}}$ is $c_B$-bi-Lipschitz equivalent to the Euclidean metric
for some  $c_B\geq 1$, see \eqref{BiLipschitz}.

From Lemma \ref{Dynamical} we know that  $X_1 = \partial B(\partial_\infty C_0^1,  e^{-c_0} e^{-s_1} /2)$ is a $(n-1)$-dimensional Euclidean sphere.
For the third case,  it follows from symmetry that $X_3 = \partial B(0,  r_3)$ is as well a $(n-1)$-dimensional Euclidean sphere.
For the second case, denote for a point $x \in \H^{n+1}$ the set of $\xi \in \R^n$ for which the penetration length of $\gamma_{x, \xi}$ in $\cN_{\e_0}(C_0)$ equals precisely $c_0$ by $S_x(c_0) = (\frak{p}_{x, C_0}^2)^{-1}(c_0) \subset \R^n$.
Since $x_2 \in C_0$,
$S_{x_2}(c_0)$ is again by symmetry a $(n-1)$-dimensional Euclidean sphere $\partial B(\eta_0,  r_2)$. 
Moreover, for $c_0$ sufficiently large, $S_x(c_0)$ is a submanifold which varies smoothly in $x$ (since $\frak{p}_{x, C_0}$ varies smoothly in $x$), showing that $X_2=S_o(c_0)$ is diffeomorphic to the Euclidean sphere $S_{x_2}(c_0)$.
Note that the visual metrics $d_{0}$ and $e^{-d(o,x_2)}\cdot d_{x_2}$ are bi-Lipschitz equivalent (with a constant independent on $d(o, x_2)= s_0^2$).
It follows from Lemma \ref{Dynamical}
that $(X_i, d_o\lvert_{X_i \times X_i})$ is $L$-bi-Lipschitz homeomorphic to a Euclidean sphere 
$\partial B(0,  r_i)$ with the induced Euclidean metric and of radius $r_i= e^{-(s_0^i + c_0)}$,
where $L\geq 1$ is independent of $s_0^i$ and $c_0$;
let $f_i : \partial B(0,  r_i) \to X_i$ denote this homeomorphism. 

Define $S_r^{n-1} \equiv \partial B(0,  r) \subset \R^n$ an Euclidean sphere of radius $r$.
For the unit sphere $S^{n-1}_1$ with the angle metric the Lebesgue measure $\mu$, restricted to balls of radius at most  $\pi/16$, clearly satisfies a power law with exponent $n-1$;
that is,
$c_l R^{n-1} \leq \mu(B(x, R)) \leq c_u R^{n-1}$ for  multiplicative constants $c_u\geq c_l>0$ and all balls $B(x, R) \subset S^{n-1}_1$ with $R\leq \pi/16$.
For $S_r^{n-1} \subset \R^n$ with the induced metric
the (radial) projection map $g_r : S^{n-1}_1 \to S^{n-1}_r$ is a $2r$-bi-Lipschitz homeomorphism,
restricted to balls of radii at most $\pi/16$ and $r/16$ respectively.
Thus, the push-forward measure $(g_r)_*\mu$ supported on $S_r^{n-1}$, restricted to balls of radius at most $r/32$,
satisfies also a power law with exponent $n-1$ and 
multiplicative constants $c_u = \bar c_u r^{n-1}$, $c_l=\bar c_l r^{n-1}$ where $\bar c_u$, $\bar c_l$ are independent of $r$.

Finally, it is readily checked that the push forward measures $\mu_i \equiv (f_i \circ g_{r_i})_*\mu$ on $(X_i, d_o\lvert_{X_i \times X_i})$ give the desired measures,
restricted to balls of radius at most $r_i/ \bar c_n$ where $\bar c_n \geq 1$ is sufficiently large, depending only on $ \bar c_u$, $\bar c_l$, $L$ and $ c_B$.
\end{proof}

Finally, we determine the following parameters for $(\mu2)$. 

\begin{lemma} 
\label{LemmaDecaying}
For $c\geq  2d_*$, 
 $(\Omega_i, \psi_i, \mu_i)$ is $\tau_l(c)$-decaying with respect to $\cal{R}^i$
where
 \be
 \nonumber
 	\tau_l(c)^1 = \tau_l(c)^2 \equiv \bar c_n^2\ e^{(n-1)(d_*+d_c)} \cdot e^{-(n-1)c},\ \ \ \ \  \tau_l(c)^3 = \bar c_n^2 \ k_0 \ e^{(n-1)(d_* + d_c)} \cdot (c+l_*^3) \cdot  e^{-(n-1)c},
\ee
and $k_0$ denotes the constant from Proposition \ref{Distribution}.
\end{lemma}

\begin{proof}
Let $\omega = (\xi, t+ d_*) \in  \Omega_i$.
For the first and second case, we know from Proposition \ref{Distribution} that distinct  $\partial_\infty C, \partial_\infty \bar C$ in $\cR^i(t -  l_*^i)$  
satisfy
\be
\nonumber
	d_o(\partial_\infty C, \partial_\infty \bar C) \geq e^{-  l_i} e^{-\max\{s_C^i, s_{\bar C}^i\}} \geq 3 \cdot e^{-t},
\ee 
since  $s_C^i, s_{\bar C}^i \leq t - l_*^i \leq t - l_i - \log(3)$.
Hence, at most one point of $\cR^i(t -  l_*^i)$ can lie in the ball $ B_{d_o}(\xi, 1.5e^{-t})$. 
In particular, for $c\geq d_* = \log(3)$, for at most one such point  $\eta$ of $\partial_\infty C, \partial_\infty \bar C$, 
the ball $B_{d_o}(\eta, e^{-(t+ c - d_*)}) \subset B_{d_o}(\eta, e^{-t}/3) $ can intersect $B_{d_o}(\xi, e^{-t}) \supset \psi(\omega)$. 
Let $\eta$ be such a point and note that the measure of $B_{d_o}(\eta, e^{-(t + c - d_*)}) \cap X_i$ is clearly 
maximized when $\eta \in X_i$.
Thus, since $(\Omega_i, \psi_i, \mu_i)$ satisfies a power law, we have
\bea
\nonumber
	\mu(\psi(\omega) \cap \cal{N}_{e^{-(t + c - d_*)}}(\cR^i(t -  l_*^i)) 
		&\leq& \mu(X_i \cap B(\eta, e^{-(t + c - d_*)} )) 
		\\ \nonumber
		&\leq& c_u e^{-(n-1)(t+c-d_*)} 
		\\ \nonumber
		&\leq& 
		\tfrac{c_u}{c_l} e^{-(n-1)(c - d_*-d_c)} \cdot c_l e^{-(n-1)(t+d_c)} \leq \tau^i_c \cdot \mu(\psi(\omega)),
\eea
showing the claim.

For the third case, consider the ball $B= B_{d_o}(\xi, 2e^{- t}) = B_{d_o}(\xi, e^{- (t- l_*^3) }) \supset \psi(\omega)$.
From Proposition \ref{Distribution} we know for all $c\geq 0$ that
\be
\nonumber
	\lvert \{\partial_\infty x \in B : x \in \cC^3 \text{ with } d(o,x)  \in [t  -c, t ]\}\rvert \leq k_0 \cdot c.
\ee
Then, with the same arguments as above, we have
\bea
\nonumber
	\mu(\psi(\omega) \cap \cN_{e^{-(t + c - d_*)})} (\cR^3(t - l_*^3, c))
		&\leq& \sum_{\partial_\infty x \in B : d(o,x)  \in [t- l_*^3 - c,  t- l_*^3] } \mu(X_3 \cap B_{d_o}(\xi_\infty, e^{-(t + c - d_*)})	
		\\ \nonumber
		&\leq& c_u\  k_0 \cdot (c+ l_*^3) \cdot   e^{-(n-1)(t + c - d_*)}
		\\ \nonumber
		&\leq& 
		\tfrac{c_u}{c_l}\ k_0 \cdot   (c+ l_*^3) \cdot   e^{-(n-1)(c - d_* - d_c)} \cdot c_l e^{-(n-1) (t+d_c)} 
		\\ \nonumber
		&\leq& 
		\tau_l(c)^3\cdot \mu(\psi(\omega)),
\eea
finishing the proof.
\end{proof}

Assume that $s_0$ (and hence $c= (s_0 + \bar \k_u - l_*^i)/2$) is sufficiently large as above and independently from $c_0$  such that $\tau_l(c)^i<1$ as well as
\be
\nonumber
	\lvert \log(1-\tau^i_c) \rvert \leq \tfrac{1}{4} \log( 2 \bar c_n^4 e^{(n-1)(d_*+d_c)}) \leq  \tfrac{1}{4} \log( \bar c_n^{4} 2^{n} 3^{n-1}).
\ee
Summarizing, when both $c_0, s_0 \geq \bar t_0$ are sufficiently large,
Proposition \ref{LBFormula} implies that for all three cases
\bea
\nonumber
	\text{dim}(\textbf{Bad}_{X_i}(\cal{C}^i, 2c  + l_*^i))
	 &\geq &
	\tau
	- \frac{  \log( 2 \bar c_n^4 e^{\tau(d_*+d_c)}) + \lvert \log(1-\tau^i_c) \rvert  }
	{ c }
	\\ \nonumber
	&\geq& (n-1) -  \frac{   \log( \bar c_n^{4} 2^{n} 3^{n-1})   }{ 2(  \frac{s_0 + \bar \k_u - l_*^i}{2})}	
	\\ \nonumber
	&\geq& (n-1) -  \frac{ \bar k_0 }{s_0 },
\eea
for a suitable constant $\bar k_0= \bar k_0(\bar t_0) >0$ independent of $c_0$, $s_0$.
This finishes the proofs.

\begin{remark}
For $n=2$, $X_1$ is a Euclidean sphere $S^1$ of radius $e^{-(s_1+c_0)}/2$ in which balls can be subpartitioned into smaller balls.
Following again \cite{Weil3}, Section $2.3.2$,  the lower bound can be improved to  $1-  \frac{ \bar k_0  }{ s_0\cdot e^{s_0}}$ for some $\bar k_0>0$.
\end{remark}


\subsubsection{Proof of Theorem \ref{ThmJTI} (Jarn\'ik-type inequality)}
\label{SectionJTI}
Assuming that  we are given the collection $\cC^3$ we let $t_*=t_0 \geq \bar c$ (as in Lemma \ref{Dynamical} and Proposition \ref{Distribution}) 
be sufficiently large and $X_3= S^n$ be the full boundary.
Since we only consider subrays $\gamma_v\lvert_{[t_0, \infty)}$, 
it is readily checked that points $x \in \cC^3$ with $d(o, x) \leq t_* - 1$ will play no role and we may hence exclude them from the collection $\cC^3$.
Let $\mu$ be the Lebesgue measure on $S^n$ for which $(\Omega, \psi, \mu)$ satisfies a power law with respect to the exponent $n$ and positive constants $c_l$, $c_u$.
We need the following lemma.

\begin{lemma}
\label{LemmaDirichlet}
For $c>  d_*$, 
 $(\Omega, \psi, \mu)$ is $\tau_l(c)$-decaying with respect to $\cal{R}$
for 
\be
\label{Decaying2}
	\tau_l(c) = \tfrac{c_u}{c_l}\ k_0 \ e^{ 2d_* n} \cdot (c+l_*) \cdot  e^{-nc},
\ee
where $l_*=l_*^3$ is as above.

Moreover,  given $B=B_{d_o}(\eta, e^{-(t-u_*- d_*)})$ where $(\eta, t-u_* - d_*) \in \Omega_3$, we have
\be
\label{Dirichlet}
	\mu( B \cap \bigcup_{\partial_\infty x \in \cR(t)} B(\partial_\infty x, e^{-(s_x+c + d_*)} )) \geq \bar k_u \ e^{-n c} \cdot \mu(B) \equiv \tau_u(c) \cdot \mu(B),
\ee
where $\bar k_u$ denotes a constant  independent on $c>0$ and
$u_*$ is the constant from Proposition \ref{Distribution}.
\end{lemma}

In the language of \cite{Weil3}, \eqref{Dirichlet} means that $(\Omega, \psi,\mu)$ is \emph{$\tau_u(c)$-Dirichlet with respect to $\cR^3$} and the parameters $(c,u_*)$.
Moreover, all the requirements are satisfied to apply the axiomatic approach of \cite{Weil3} for the upper bound to our setting.

\begin{proof}[Proof of Lemma \ref{LemmaDirichlet}]
The first part follows in a similar way to the proof of Lemma \ref{LemmaDecaying}.

For the second part, given $B$ from the statement, Proposition \ref{Distribution} shows that there exists a point $x \in \cC^3$ such that 
$\partial_\infty x \in B_{d_o}(\eta, e^{-(t-u_*)})$ with $t-u_* \leq s_x \leq t$.
Hence, by definition of $d^c$, we have $B_{d_o}(\partial_\infty x, e^{- (s_x + c )}) \subset B_{d_o}(\partial_\infty x, e^{- (t- u_* + c )}) \subset B$.
Thus, we see 
\bea
\nonumber
	\mu( B \cap \bigcup_{\partial_\infty y \in \cR(t)} B(\partial_\infty y, e^{-(s_x+c + d_*)} )) &\geq& \mu(B_{d_o}(\partial_\infty x, e^{- (s_x + c + d_*)}) )
	\\ \nonumber
	&\geq& c_l  e^{- n(s_x + c +d_*)} 
	\\ \nonumber
	&\geq& \tfrac{c_l}{c_u}   e^{n ( (t -s_x) - (c+  u_* + d^c + d_*))}   \cdot c_u e^{-n (t-u_*- d_*)} 
	\\ \nonumber
	&\geq& \tfrac{c_l}{c_u}   e^{- n(c  + u_* + 2d_*) }   \cdot c_u e^{-n (t-u_*- d_*)} 	
	\geq \tau_u(c)\cdot \mu(B),
\eea
as claimed.
\end{proof}

Clearly, since $t_* = t_0$ and $\min\{s_x : x \in \cC^3\} \geq t_* -1$, $(S0)$ is trivially satisfied for all $\omega_0=(\xi, t_*) \in \Omega$ when restricting to $c \geq 1$.
We may assume that $\xi=0 \in \R^n$ and let $c_B\geq 1$ be the constant such that $d_o$, restricted to the unit ball in $\R^n$, 
is  bi-Lipschitz equivalent to the Euclidean metric $d_E$.
Since $t_*$ is sufficiently large, we may assume that $B_{d_o}(0, 2e^{-t_*}) \subset B(0,1)$ is contained in the Euclidean unit ball.

In particular, \eqref{Decaying} and \eqref{Dirichlet} also hold for the Euclidean metric and the Lebesgue measure on $\R^n$, up to a multiplicative constant depending on $c_B$.
Finally,  we  remark that \eqref{Decaying} and \eqref{Dirichlet} also remain true, 
up to a further multiplicative constant (that is with respect to $\bar \tau_l(c) = k_l \cdot c\cdot e^{-n c}$ and $\bar \tau_u(c) = k_u \cdot e^{-nc}$ respectively),
if we replace the Euclidean balls $B(x, r)$ by Euclidean cubes $Q(x,r) = [x_1 - r, x_1+r] \times \dots \times  [x_n - r, x_n+r]$, 
centered at $x = (x_1, \dots, x_n) \in \R^n$ with the same radius $r>0$;
see for instance Lemma $2.6$  in \cite{Weil3}.
Thus, for $N_0 \leq m\in \N$ with $N_0$ sufficiently large 
we let $c=\log(m)$ or $c+ u_c= \log(m)$ for some $u_c$  such that  $u^* \leq u_c\leq u_* + \log(2)$
and can in fact apply Proposition $2.5$ combined with Theorem $2.1$ of \cite{Weil3}.
From these we obtain
\bea
\nonumber
	\text{dim}(\textbf{Bad}_{S^n}(\cal{C}^3, 2c + l_*^3) ) 
	&\geq& 
	 n  	- \frac{ \lvert   \log(1- \bar \tau_l(c)) \rvert  }{ c } 
	 \\ \nonumber
	 &= &n  	- \frac{ \lvert   \log(1-  k_l \cdot c \cdot e^{-n c}) \rvert  }{ c } 
\eea
as well as
\bea
\nonumber
	\text{dim}(\textbf{Bad}_{S^n}(\cal{C}^3, c ) \cap B_{d_o}(\xi, e^{-t_*}) ) 
	& \leq &  n - \frac{ \lvert   \log(1-\bar \tau_u(c) )\rvert}{c + u^c } )  
	\\ \nonumber
	&\leq& n - \frac{ \lvert   \log(1-  k_u \cdot e^{-n c}  )\rvert}{c+   u^* + \log(2) }.
\eea
Since the argument is independent from the chosen formal ball $\omega_0=(\xi, t_*)$ and 
using the countable stability of the Hausdorff-dimension, the upper bound holds for $\textbf{Bad}_{S^n}(\cal{C}, c )$.
Finally, applying the Taylor expansion and using Lemma \ref{Dynamical} finishes the proof.

\subsection{The absolute winning game}
\label{SectionAbsWinningGame}
Recall again the setting and abstract framework introduced in Section \ref{SecSetting}.
In Section \ref{SectionAWG} we define the absolute winning game and deduce abstract conditions under which \textbf{Bad}$_X(\cC)$ (for a collection $\cC$ as in Section \ref{SectionSetting}) is absolute winning.
Then we verify these conditions and apply the result for the respective settings in Sections \ref{SectionAbsWinning} and \ref{SectionPointAbsWinning}.

\subsubsection{The absolute winning game and conditions for winning sets}
\label{SectionAWG}
Given a closed subset $X \subset \bar X$ of a proper metric space $\bar X$, let us define the \emph{absolute game} on $X$, relative to $\bar X$.
Let $\b_*>0$ be a fixed parameter and choose any $0<\b<\b_*$.
When $X= \bar X= \R^n$ and $\b_*=1/3$, then the game below corresponds to the classical \emph{absolute game} introduced by McMullen \cite{McMullen}.
Consider two players, Alice and Bob.
Bob starts by choosing a metric ball $B_0 = B(x_0, r_0)$ centered at $x_0 \in X$.
At the $k$.th step of the game, $k\geq 0$, assume that Bob has chosen his ball $B_k = B(x_k, r_k)$ with $x_k \in X$.
Alice is then allowed to \emph{block} a ball $A_k= B(y_k, \a_k r_k)$ centered at any point $y_k \in \bar X$ with the restriction that $0 < \a_k < \b$.
Again, Bob continues by choosing a ball $B_{k+1} = B(x_{k+1}, r_{k+1})$ centered at $x_{k+1} \in X$ and with $\b r_k  \leq r_{k+1} \leq r_k$.
The game continues in this manner and we obtain a sequence 
\be
\nonumber
	B_0 \supset B_0 - A_0 \supset B_1 \supset \ldots \supset B_k - A_k \supset B_{k+1} \supset \ldots
\ee
A given \emph{target set} $S \subset X$ is said to be \emph{$\b$-absolute winning} in $X$ if Alice has a strategy to guarantee that the intersection $\cap_{k\geq0} B_k$ intersects $S$ nontrivially, no matter of Bob's choices.
If $S$ is $\b$-absolute winning for all $\b < \b_*$, then $S$ is \emph{absolute winning} in $X$ (given $\beta_*$).\\

Let $\cC$ be a collection which defines a set \textbf{Bad}$_{X}(\cC)$ of badly approximable points in $X$ as in Section \ref{SectionAbstractFramework}.
We are interested in conditions on the space $\bar X$, the subspace $X\subset \bar X$ and the collection $\cC$ such that the
target set $\textbf{Bad}_{X}(\cC)$ is an absolute winning set in $X$.

For a metric ball $B(x, r)$ we let $c \cdot B(x,r)=B(x, c\cdot r)$ for $c>0$ and we set $B_{\bar X}(x, t) \equiv B(x, e^{-t})$ for $(x,t) \in \bar \Omega=\bar X \times (t_*, \infty)$.
Fix $b_*\geq0$ and assume that $\bar X$ and $\cC$ satisfy the following.

\begin{itemize}
\item[ \text{[$N(b)$]} ] Given $b>b_*$  there is an integer $N(b) \in \N_{\geq 2}$ such that, given any ball $B=B_{\bar X}(x, t)$ for $(x, t) \in \Omega_{\bar X}$, 
there is a collection of $N(b)$ formal balls 
\be
\label{AxiomN}
	\cC(B)=\{ \omega_i =(x_i, t + b) \in \Omega_{\bar X} : i =1, \ldots, N(b) \}
\ee
such that
\be
\nonumber
	2 \cdot B \subset \bigcup_{\omega_i \in \cC(B)} \tfrac{1}{2} \cdot B_{\bar X}(\omega_i).
\ee

\item[ \text{[$\varphi$]} ] there is a non-decreasing function $\varphi: (b_*, \infty) \to \R$ such that, 
given any ball $B=B_{\bar X}(x, t)$ for $(x, t) \in \Omega_{\bar X} $, we have (independently from $t$)
\be
\nonumber
	\lvert 2\cdot B \cap \cR(t, b) \rvert \leq \varphi(b).
\ee
\end{itemize}

Now, given a closed subset $X\subset \bar X$, we require the following condition on $X$. 
Recall from \cite{BroderickEtAl} that $X$ is \emph{$b_*$-diffuse} in $\bar X$
if, given any points $\bar x \in \bar X$ and $x \in X$, $t \geq t_0$ (some fixed $t_0$), there exists a further point $y \in X$ such that
\be
\nonumber
	B(y, e^{-(t+b_*)}) \subset B(x, e^{-t}) - B(\bar x, e^{-(t+b_*)}).
\ee

\begin{example}
If $(\Omega_{X}, B_{ X}, \mu)$ (respectively, $(\Omega_{\bar X}, B_{\bar X}, \mu)$) satisfies a power law then there is a $b_*>0$ such that $X$ is $b_*$-diffuse in $\bar X$ and (respectively, we have $[N(b)])$).
If $X$ is an embedded submanifold (of dimension at least one) in a Riemannian manifold $\bar X$ then both $X$ and $\bar X$ satisfy a power law for the natural volume measures $\mu_X$ and $\mu_{\bar X}$ respectively.

Furthermore,  condition $[\varphi]$ is satisfied by all the collections $\cC^i$ given above by Proposition \ref{Distribution}.
\end{example}

The above conditions are used to show the following theorem.

\begin{theorem}
\label{ThmAbsWinningAbstract}
Let $X $ be $b_*$-diffuse in $\bar X$. 
Assume that $\bar X$ satisfies $[N(b)]$ and $\cC$ satisfies $[\varphi]$ for a function $\varphi$ of at most polynomial growth.
Then \textbf{Bad}$_X(\cC)$ is absolute winning in $X$.
\end{theorem}

\begin{remark}
The assumption that $\varphi$ is of at most polynomial growth can be relaxed to the one of \eqref{ChoiceN} below.
\end{remark}

\begin{proof}
For a given parameter $b>b_*$ let $N(b) \geq 2$ and $\varphi$ be provided by $[N(b)]$ and $[\varphi]$ respectively.
Let $B_0=B(x_0, e^{-t_0})$, $x_0 \in X$, be the first choice of Bob.
Since $\varphi$ is at most of polynomial growth, choose $n=n(b, t_0) \in \N$ such that
\be
\label{ChoiceN}
	(\frac{N(b)}{N(b) - 1})^{n-1} \geq \varphi(nb) \ \ \ \text{ and } \ \ \ nb \geq t_0 - s_0.
\ee

We set up the following strategy for the parameter $b>b_*$.
Let $B=B_{kn + i}$ with $k\geq 0$ and $0\leq i <n$ be the ball chosen by  Bob at the $(kn+i)$-th. step.
Set $R_k = \cal{R}(t_{kn}, nb) $ and define 
\be
\nonumber
	Z_k^i \equiv | R_k \cap \big( 2\cdot B - \bigcup_{l < kn+ i} \tfrac{1}{2} \cdot A_l \big) | ,
\ee
where $A_l$ are the choices of Alice in the previous steps.
The strategy for Alice is to block in $n$ steps the relevant and at most $Z_k^0$ points of $R_k \cap 2\cdot B_{kn}$, where 
by assumption 
\be
\nonumber
	Z_k^0 \leq | 2 \cdot B_{kn} \cap R_k | \leq \varphi(nb).
\ee
First, let $i<n-1$.
Cover $2\cdot B$ by $N(b)$ balls $B_j = B(y_j, e^{-(t_{kn+i} + b)}/2)$ of the collection $\cC(B)$ provided by \eqref{AxiomN}.
One of the balls, say $B_{j_0}$, must contain at least $\lceil Z_k^i/N(b) \rceil$ of the points of $ 2 \cdot B  \cap R_k$.
In other words, if we set
\be
\nonumber
	A_{nk+i} \equiv B(y_{j_0}, e^{-(t_{kn+i} + b)}) = 2 B_{j_0}, 
\ee
then for any ball $\tilde B=B_{kn+(i+1)}=B(x, e^{-t_{kn+(i+1)}})$, $x \in X$, chosen by Bob with $\tilde B \subset B - A_{nk+i}$ (which Bob is able to since $X$ is $b_*$-diffuse)
we have
\bea
\nonumber
	Z_k^{i+1} &=& | R_k \cap \big( 2\cdot \tilde B - \bigcup_{l < kn+ i+1} \tfrac{1}{2} \cdot A_l  \big) | 
	\\ \nonumber
	 &\leq& | R_k \cap \big( 2\cdot B - \bigcup_{l < kn+ i} \tfrac{1}{2} \cdot  A_l \big) | - \lceil Z_k^i/N(b) \rceil   \leq Z_k^i(1- N(b)^{-1}).
\eea
If $i=n-1$, then by induction and \eqref{ChoiceN} we have that
\be
\nonumber
	Z_k^{n-1} \leq Z_k^0(1 - N(b)^{-1})^{n-1} \leq \varphi(nb) (\frac{N(b) - 1}{N(b)})^{n-1} \leq 1.
\ee
Thus, set
\be
\nonumber
	A_{kn + (n-1)} = B(x, e^{- (t_{kn+(n-1)} + b)})
\ee
if $R_k \cap \big( 2\cdot B - \bigcup_{l < (k+1)n} \tfrac{1}{2} A_l \big) = \{x\}$, and otherwise let $A_{kn + (n-1)} $ be empty.
This finishes the strategy.

Finally, we claim that the above strategy is winning. 
Let 
\be
\nonumber
	x \in B_{\infty}=\cap_{k\geq 0} B_k ,
\ee
with $x \in X$ (which exists since $X$ is complete).
Let $z \in \cal{C}$. Since $nb \geq t_0 - s_0$ we have $z \in \cup_{k\geq 0} R_k$ and if  $z \in R_k$ then $s_z \in (t_{kn} - nb, t_{kn}]$.
By the above strategy we must have that either $z \not \in 2\cdot B_{kn + (n-1)}$
or $z \in \cup_{l \leq kn + (n-1)} \tfrac{1}{2} A_l$ and hence 
\be
\nonumber
	B(z, e^{- (t_{kn+(n-1)} + b)}/2 )\subset \bigcup_{l \leq kn + (n-1)} A_l.
\ee
In both cases, $d(x, z) \geq e^{-t_{(k+1)n}}/2 \geq e^{-s_z} e^{-2nb}/2$, showing that $z \in \textbf{Bad}_X(\cC)$ and  finishing the proof.
\end{proof}

\subsubsection{Proof of Theorem \ref{ThmAbsWinning} (Absolute winning)}
\label{SectionAbsWinning}
Let $M$ be as in Section \ref{SectionHeight} which determines a  collection $\cC^1$ of pairwise disjoint horoballs as in Section \ref{SectionSetting}.
Identify $\partial C_o^1 \subset \H^{n+1}$ with $\R^n$ via the map $x \mapsto \gamma_x(\infty) \in \R^n$ where $\gamma_x$ is the vertical geodesic line in $\H^{n+1}$ with $\gamma_x(0) = x \in \partial C_o^1$.
Note that the cocompact torsion-free stabilizer $\Gamma_\infty = $ Stab$_\Gamma(\infty)$ acts isometrically on $\R^n$ and since $C_o^1$ is precisely invariant,
the projection $\tilde \pi$ of a compact fundamental domain $F$ of $\Gamma_\infty$ in $\partial C_o^1 = \R^n$ locally embeds isometrically into $\partial H_0 \subset M$ 
(up to rescaling the length metric on $\partial H_0$).
In particular, identify $\R^n/\Gamma_\infty$  with $SH_0^+$ via the map
\be
\nonumber
	F\ni x \mapsto d \tilde \pi (\dot \gamma_x(0)) \equiv v_x \in SH_0^+ 
\ee
which together with the translation to $F$ determines the map in \eqref{CanonicalMap}.

Let $c_0 >0$ and recall that  $\tilde S_{c_0} \subset \partial C_o^1 = \R^n$ denotes  the lift of $S_{c_0}$, 
the set of vectors in $SH_0^+$ for which the first penetration height equals $c_0$.
From Lemma \ref{Dynamical}, $\gamma_x$, $x\in \R^n$, intersects a horoball $C$ in $\cC^1$ with height $c_0$ if and only if $\gamma_x(\infty)$ is contained in the $(n-1)$-dimensional Euclidean sphere $S_C^{c_0} \equiv \partial B(\partial_\infty C, e^{-c_0} e^{-s_C}/2)$.
Since $c_0>0$, if $\gamma_x(\infty) \in S_C^{c_0} \cap S_{\bar C}^{c_0}$ for $C, \bar C \in \cC^1$ distinct and $s_C \leq s_{\bar C}$,
then by the disjointness of $C$ and $\bar C$ the ray $\gamma_x$ hits $C$ at a smaller time that it hits $\bar C$.
Thus we see that $\tilde S_{c_0}$ consists of a countable disjoint  union of $(n-1)$-dimensional Euclidean spheres in $\R^n$.

We need the following.

\begin{lemma}
Every $S=S_C^{c_0}$ is $\log(3)$-diffuse in $\R^n$ for $t_0 = s_C +  c_0 + 2\log(2)$.
\end{lemma}
 
\begin{proof}
This follows either since $S$ is a smoothly embedded submanifold and by the above remark or can be seen as follows.
Consider two points $x \in S$ and $\bar x \in \R^n$. 
Given $t\geq t_0$, let $\bar B = B(\bar x, e^{-t}/3)$.
Clearly, the worst case to consider is when $\bar x$ actually lies on $S$.
But for this case, since $e^{-t} \leq e^{-t_0} = r/4 $ for the radius $r$ of  $S$, 
 it is easy to see that there exists a point $y \in S$ such that $B(y, e^{-t}/3) \subset B(x, e^{-t}) - \bar B$, finishing the proof.
\end{proof}

Thus, we are given a $b_*$-diffuse set $X = S_C^{c_0} \subset \R^n = \bar X$.
Clearly, $\R^n$ satisfies a power law for the Lebesgue measure, hence $[N(b)]$ is satisfied, and we have $[\varphi]$ for a function $\varphi$ of at most polynomial growth by the first part of Proposition \ref{Distribution}. 
Theorem \ref{ThmAbsWinningAbstract} implies that $\textbf{Bad}_{S_C^{c_0}}(\cC^1)$ is absolute winning in $ S_C^{c_0}$.
This, together with Lemma \ref{Dynamical} which shows that $\tilde{\cal{B}}_{h_1=c_0} \cap S_C^{c_0} =  \textbf{Bad}_{S_C^{c_0}}(\cC^1)$, already finishes the proof of Theorem \ref{ThmAbsWinning}.

\begin{remark}
Note that the first part of Proposition \ref{Distribution} and Lemma \ref{Dynamical}  also hold in curvature at most $-1$.
Moreover the above arguments translate in a similar way (if we replace spheres by submanifolds diffeomorphic to spheres) so that Theorem \ref{ThmAbsWinning} holds for pinched negative curvature as well.

Also, since $c(I)$ is $b_*$-diffuse (as remarked above) for any smoothly embedded (non-constant) curve $c:I \to \R^n$,  
it follows by the above arguments that $\cal{\cB}_{M, e, \b_e} \cap c(I)=  \textbf{Bad}_{c(I)}(\cC^1)$ is absolute winning in $c(I)$.
\end{remark}

\subsubsection{Proof of Theorem \ref{PointAbsWinning} (Absolute winning)}
\label{SectionPointAbsWinning}
Recall that $M = \H^{n+1}/\Gamma$ is hyperbolic and geometrically finite and we consider the setting of Section \ref{SectionSetting} for the collection $\cC^3$ of $\tau_0$-separated points.
The set  of positively recurrent vectors at $o$ can be identified with the conical limit set $\Lambda\Gamma_c$  of $\Gamma$.
Since $\Gamma$ is non-elementary geometrically finite, the limit set $\Lambda \Gamma$ consists of the conical limit set $\Lambda\Gamma_c$ and a countable union of parabolic fixed points, where $\Lambda\Gamma_c$ is dense in $\Lambda\Gamma$; see \cite{Nicholls}.
Thus, the closure $X\equiv \overline{SM_o^+}$ can be identified with the limit set $\Lambda\Gamma$ in $\partial_\infty \H^{n+1}=S^{n}\equiv \bar X$. 

\begin{lemma}
The  limit set $\Lambda \Gamma$ is $b_*$-diffuse in $S^n$ for some $b_*>0$.
\end{lemma}

\begin{proof}
Note that the limit set $\Lambda \Gamma$ of a non-elementary geometrically finite Kleinian group is uniformly perfect in $S^n$ (see \cite{JarviVuorinen}), which follows to be $b_*$-diffuse in $S^n$  for some $b_*>0$ by \cite{MayedaMerrill}, Lemma $2.4$.
\end{proof}

Clearly, $S^{n}$ satisfies a power law for the spherical measure, hence $[N(b)]$ is satisfied, and we have $[\varphi]$ for a function $\varphi$ of at most polynomial growth by the third part of Proposition \ref{Distribution}. 
Theorem \ref{ThmAbsWinningAbstract} implies that $\textbf{Bad}_{\Lambda\Gamma}(\cC^3)$ is absolute winning in $\Lambda\Gamma$,
where by Lemma \ref{Dynamical} we have $\cB_{M,o,x_0, t_0}  =  \textbf{Bad}_{\Lambda\Gamma}(\cC^3)$.

Finally, if $M$ is convex-cocompact or of finite volume then the Patterson-Sullivan measure at $o$ satisfies a power law with exponent $\delta= \text{dim}(\Lambda\Gamma)$.
 It follows that $\cB_{M,o,x_0, t_0}$ is thick in $\overline{SM_o^+}$, see for instance \cite{Weil2}.
 This finishes the proof of Theorem \ref{PointAbsWinning}.

\begin{remark}
When $M$ is of finite volume, hence $ \Lambda\Gamma=S^n = SM_o$, then as remarked above any smoothly embedded (non-constant) submanifold of $S^n$ is $b_*$-diffuse in $S^n$. In particular, it follows by the above arguments that $\cB_{M,o,x_0, t_0} \cap c(I)=  \textbf{Bad}_{c(I)}(\cC^3)$ is absolute winning in $c(I)$ for any smoothly embedded (non-constant) curve $c:I \to SM_o$.
\end{remark}


\section{Acknowledgments.}
The author is thankful to Barak Weiss for helpful discussions and for his question which was the starting point of this work.
Finally, he acknowledges the support by the Swiss National Science Foundation Project 13509, as well as the ERC starter grant DLGAPS 279893.


\bibliographystyle         {plain}      
\bibliography{cup_ref.bib}

 \end{document}